\documentclass[12pt,a4paper]{smfart}

\usepackage[francais]{babel}
\usepackage{smfenum}
\usepackage{amsmath,amssymb,amsthm,latexsym,MnSymbol}
\usepackage[utf8]{inputenc}
\usepackage[T2A,T1]{fontenc}
\usepackage{url}
\usepackage{tikz,tikz-cd}
\usepackage[margin=2cm]{geometry}

\newcommand{\Z}{\mathbf{Z}}
\newcommand{\Q}{\mathbf{Q}}
\newcommand{\R}{\mathbf{R}}
\newcommand{\CC}{\mathbf{C}}
\newcommand{\h}{\mathcal{H}}
\newcommand{\dd}{\mathrm{d}}
\newcommand{\hol}{\mathrm{hol}}

\DeclareMathOperator{\Eis}{Eis}
\DeclareMathOperator{\GL}{GL}
\DeclareMathOperator{\id}{id}
\DeclareMathOperator{\Res}{Res}
\DeclareMathOperator{\SL}{SL}

\DeclareSymbolFont{cyrillic}{T2A}{cmr}{m}{n}
\DeclareMathSymbol{\BB}{\mathalpha}{cyrillic}{193}
\newtheorem{thm}{Théorème}[section]
\newtheorem{lem}[thm]{Lemme}
\newtheorem{pro}[thm]{Proposition}
\theoremstyle{definition}
\newtheorem{definition}[thm]{Définition}
\theoremstyle{remark}
\newtheorem{remarque}[thm]{Remarque}

\begin{document}

\title[Régulateurs modulaires]{Régulateurs modulaires explicites\\ via la méthode de Rogers-Zudilin}

\author[F. Brunault]{François Brunault}
\email{francois.brunault@ens-lyon.fr}
\address{ÉNS Lyon, UMPA\\ 46 allée d'Italie\\ 69007 Lyon, France}

\thanks{}

\begin{abstract}
Nous calculons le régulateur des éléments de Beilinson-Deninger-Scholl en termes de valeurs spéciales de fonctions $L$ de formes modulaires en utilisant la méthode de Rogers-Zudilin.
\end{abstract}

\maketitle

\section{Introduction}\label{sec intro}

Soit $N \geq 3$ un entier. Soit $Y(N)$ la courbe modulaire ouverte et $E$ la courbe elliptique universelle sur $Y(N)$. Pour un entier $n \geq 0$, notons $E^n$ la puissance fibrée $n$-ième de $E$ au-dessus de $Y(N)$. Pour tout $u \in (\Z/N\Z)^2$, avec $u \neq 0$ si $n=0$, Beilinson a défini le symbole d'Eisenstein
\begin{equation*}
\Eis^n(u) \in H^{n+1}_{\mathcal{M}}(E^n,\Q(n+1)).
\end{equation*}
Le symbole d'Eisenstein est à la base de la preuve par Beilinson et Deninger-Scholl de la conjecture de Beilinson pour les valeurs non critiques des fonctions $L$ des formes modulaires.

Plus précisément, donnons-nous un entier $k \geq 0$, et choisissons une décomposition $k=k_1+k_2$ avec $k_1,k_2 \geq 0$. Considérons les projections canoniques
\begin{equation*}
\begin{tikzcd}
& E^{k_1+k_2} \ar[swap]{dl}{p_1} \ar{dr}{p_2} & \\
E^{k_1} &  & E^{k_2}.
\end{tikzcd}
\end{equation*}
Soient $u_1,u_2 \in (\Z/N\Z)^2$, avec $u_i \neq 0$ si $k_i=0$. Généralisant les constructions de Beilinson et Deninger-Scholl, Gealy \cite{gealy:thesis} définit un élément dans la cohomologie motivique de $E^k$
\begin{equation}\label{Eisk1k2}
\Eis^{k_1,k_2}(u_1,u_2) = p_1^* \Eis^{k_1}(u_1) \cup p_2^* \Eis^{k_2}(u_2) \in H^{k+2}_{\mathcal{M}}(E^k,\Q(k+2)).
\end{equation}
Dans le cas $k=0$, on retrouve les éléments de Beilinson-Kato (cup-produits d'unités de Siegel) dans le $K_2$ de la courbe modulaire $Y(N)$. Considérons le régulateur de Beilinson à valeurs dans la cohomologie de Deligne-Beilinson
\begin{equation*}
r_{\BB} : H^{k+2}_{\mathcal{M}}(E^k,\Q(k+2)) \to H^{k+2}_{\mathcal{D}}(E^k/\R,\R(k+2))
\end{equation*}
et l'application analogue pour la compactification lisse $\overline{E}^k$ de $E^k$ définie par Deligne. Soit $f$ une forme parabolique primitive de poids $k+2$ pour le groupe $\Gamma_1(N)$. Notons $K_f$ le corps des coefficients de $f$, et $\omega_f \in \Omega^{k+1}(\overline{E}^k) \otimes K_f$ la forme différentielle associée à $f$. Classiquement, le régulateur de Beilinson associé à $f$ est défini au moyen de l'accouplement issu de la dualité de Poincaré
\begin{equation*}
\langle \cdot , \omega_f \rangle : H^{k+2}_{\mathcal{D}}(\overline{E}^k/\R,\R(k+2)) \to K_f \otimes \CC.
\end{equation*}
En utilisant les éléments (\ref{Eisk1k2}), Deninger-Scholl \cite{deninger-scholl} et Gealy \cite{gealy:thesis} montrent qu'il existe un élément $x$ de $H^{k+2}_{\mathcal{M}}(\overline{E}^k,\Q(k+2)) \otimes K_f$ tel que
\begin{equation*}
\langle r_{\BB}(x), \omega_f \rangle = \Omega_f^+ L'(f,0)
\end{equation*}
où $\Omega_f^+$ est la période réelle de $f$. Le calcul est basé sur la méthode de Rankin-Selberg. Cela permet de démontrer la conjecture de Beilinson pour la valeur spéciale $L(f,k+2)$.

Dans cet article, nous proposons une approche nouvelle, et totalement explicite, pour calculer le régulateur. Au lieu d'intégrer le régulateur de Beilinson contre une forme parabolique, nous pouvons l'intégrer le long des $(k+1)$-cycles explicites fournis par la théorie de Shokurov \cite{shokurov}. Plus précisément, considérons le cycle de Shokurov $X^k \{0,\infty\}$ (voir la section \ref{sec shokurov}). Dans la section \ref{deninger-scholl-gealy}, nous définissons une forme différentielle explicite $\Eis^{k_1,k_2}_{\mathcal{D}}(u_1,u_2)$ représentant le régulateur de Beilinson de $\Eis^{k_1,k_2}(u_1,u_2)$. Considérons alors l'intégrale
\begin{equation}\label{int EisD}
\int_{X^k \{0,\infty\}} \Eis^{k_1,k_2}_{\mathcal{D}}(u_1,u_2).
\end{equation}
Contrairement au cas $k=0$, l'intégrale (\ref{int EisD}) ne converge pas absolument en général. Pour remédier à ce problème, nous introduisons un paramètre complexe $s \in \CC$ et considérons l'intégrale
\begin{equation}\label{int EisD s}
\int_{X^k \{0,\infty\}} \Im(\tau)^s \Eis^{k_1,k_2}_{\mathcal{D}}(u_1,u_2).
\end{equation}
Nous montrons que l'intégrale (\ref{int EisD s}) converge pour $\Re(s) \ll 0$ et se prolonge en une fonction méromorphe sur $\CC$, holomorphe en $s=0$. Définissons alors l'intégrale régularisée
\begin{equation}\label{int EisD *}
\int_{X^k \{0,\infty\}}^* \Eis^{k_1,k_2}_{\mathcal{D}}(u_1,u_2)
\end{equation}
comme la valeur en $s=0$ de cette fonction.

Soit $\h$ le demi-plan de Poincaré. Pour $\ell \geq 1$, $a, b \in \Z/N\Z$, définissons la série d'Eisenstein
\begin{equation*}
G^{(\ell)}_{a,b}(\tau) = a_0(G^{(\ell)}_{a,b}) + \sum_{\substack{m,n \geq 1 \\ m \equiv a, n \equiv b (N)}} m^{\ell-1} q^{mn} +(-1)^{\ell} \sum_{\substack{m,n \geq 1 \\ m \equiv -a, n \equiv -b (N)}} m^{\ell-1} q^{mn} \qquad (\tau \in \h, q=e^{2i\pi\tau})
\end{equation*}
avec
\begin{equation*}
a_0(G^{(1)}_{a,b}) = \begin{cases} 0 & \textrm{si } a=b=0\\
\frac12-\{\frac{b}{N}\} & \textrm{si } a=0 \textrm{ et } b \neq 0\\
\frac12-\{\frac{a}{N}\} & \textrm{si } a \neq 0 \textrm{ et } b=0\\
0 & \textrm{si } a \neq 0 \textrm{ et } b \neq 0,
\end{cases}
\end{equation*}
et pour $\ell \geq 2$,
\begin{equation*}
a_0(G^{(\ell)}_{a,b}) = \begin{cases} - N^{\ell-1} \frac{B_\ell(\{\frac{a}{N}\})}{\ell} & \textrm{si } b=0\\
0 & \textrm{si } b \neq 0,
\end{cases}
\end{equation*}
où $\{x\} = x- \lfloor x \rfloor$ désigne la partie fractionnaire de $x$, et $B_\ell$ est le $\ell$-ième polynôme de Bernoulli. La fonction $G^{(\ell)}_{a,b}$ est une forme modulaire (quasi-modulaire si $\ell=2$) de poids $\ell$ pour le groupe $\Gamma_1(N^2)$.

Étant donnée une forme modulaire $F = \sum_{n =0}^{\infty} a_n q^n$ pour $\Gamma_1(N)$, notons
\begin{equation*}
\Lambda(F,s) = N^{s/2} (2\pi)^{-s} \Gamma(s) \sum_{n=1}^\infty a_n n^{-s}
\end{equation*}
la fonction $L$ complétée de $F$, et notons $\Lambda^*(F,0)$ la valeur régularisée de $\Lambda(F,s)$ en $s=0$ (voir la définition \ref{def Lambdaf0}).

En utilisant la méthode de Rogers-Zudilin, nous montrons le résultat suivant.

\begin{thm}\label{main thm}
Soit $k \geq 0$ un entier, et soient $k_1,k_2 \geq 0$ tels que $k=k_1+k_2$. Soit $N \geq 3$ un entier, et soient $u_1=(a_1,b_1)$, $u_2=(a_2,b_2) \in (\Z/N\Z)^2$. Supposons $u_i \neq (0,0)$ si $k_i=0$, et $b_i \neq 0$ si $k_i=1$. Alors
\begin{equation}
\label{main formula} \int_{X^k \{0,\infty\}}^{*} \Eis^{k_1,k_2}_{\mathcal{D}}(u_1,u_2) = \frac{(k_1+2)(k_2+2)}{2N^{k+2}} (2\pi)^{k+1} i^{k_1-k_2+1} \Lambda^* \left( G^{(k_2+1)}_{b_2,a_1} G^{(k_1+1)}_{b_1,-a_2} - G^{(k_2+1)}_{b_2,-a_1} G^{(k_1+1)}_{b_1,a_2},0 \right).
\end{equation}
\end{thm}

\begin{remarque}
\begin{enumerate}
\item Lorsque $k_1=k_2=0$, on retrouve la formule pour le régulateur du cup-produit de deux unités de Siegel \cite{brunault:reg_siegel,zudilin}.
\item La forme modulaire de poids $k+2$ apparaissant dans le membre de droite de (\ref{main formula}) est à coefficients rationnels. Il est naturel de se demander si toute forme parabolique primitive est combinaison linéaire de telles formes modulaires.
\end{enumerate}
\end{remarque}

Cet article est organisé comme suit. Dans les sections \ref{sec notations} et \ref{sec EK}, nous introduisons les notations concernant la fonction zêta de Hurwitz et les séries d'Eisenstein-Kronecker, et rappelons les résultats principaux les concernant. Dans les sections \ref{sec shokurov} et \ref{sec Eis}, nous rappelons la définition des cycles de Shokurov et la formule donnant la réalisation du symbole d'Eisenstein. Dans la section \ref{deninger-scholl-gealy}, nous définissons les éléments de Deninger-Scholl et explicitons leurs réalisations en cohomologie de Deligne-Beilinson. Dans la section \ref{sec RZ}, nous exposons la méthode de Rogers-Zudilin dans un cadre suffisamment général. Dans la section \ref{sec fourier}, nous calculons le développement de Fourier de certaines séries d'Eisenstein analytiques réelles. Enfin, nous effectuons le calcul proprement dit du régulateur dans la section \ref{sec calcul reg}.

Ce travail doit beaucoup à Anton Mellit et Don Zagier, qui ont les premiers proposé une formulation analytique générale de l'astuce de Rogers-Zudilin. Je remercie Anton Mellit de m'avoir invité à l'université de Köln en novembre 2011, et pour les discussions très constructives que nous avons eues. Je remercie Odile Lecacheux pour des échanges très stimulants autour de ces questions, Loïc Merel pour avoir porté mon attention sur les cycles de Shokurov, Wadim Zudilin pour ses encouragements lors de la rédaction, et Michael Neururer pour ses commentaires et la relecture de ce texte.

\section{Fonction zêta de Hurwitz}\label{sec notations}

Pour $x \in \R/\Z$, on définit la \emph{fonction zêta de Hurwitz}
\begin{equation*}
\zeta(x,s) = \sum_{\substack{y>0 \\ y \equiv x(1)}} y^{-s} \qquad (\Re(s)>1)
\end{equation*}
et la \emph{fonction zêta périodique}
\begin{equation*}
\hat{\zeta}(x,s) = \sum_{n=1}^\infty e^{2i\pi nx} n^{-s} \qquad (\Re(s)>1).
\end{equation*}
 On a donc $\zeta(0,s)=\hat{\zeta}(0,s)=\zeta(s)$. La fonction $s \mapsto \zeta(x,s)$ se prolonge en une fonction méromorphe sur $\CC$, avec un unique pôle simple en $s=1$, de résidu égal à $1$ \cite[Corollary 2 (a)]{knopp-robins}.

\begin{definition}
Pour $x \in \R/\Z$, on pose
\begin{equation*}
\zeta^*(x,1) = \lim_{s \to 1} \left(\zeta(x,s)-\frac{1}{s-1}\right).
\end{equation*}
\end{definition}

Pour $x \in \R/\Z$, $x \neq 0$, la fonction $s \mapsto \hat{\zeta}(x,s)$ se prolonge en une fonction holomorphe sur $\CC$ \cite[Corollary 2 (c)]{knopp-robins}. Pour tout entier $N \geq 1$, en notant $\zeta_N = e^{2i\pi/N}$, on a les relations suivantes
\begin{equation*}
\sum_{x \in \Z/N\Z} \zeta_N^{xu} \zeta(\frac{x}{N},s) = N^s \hat{\zeta}(\frac{u}{N},s)
\end{equation*}
\begin{equation*}
\sum_{x \in \Z/N\Z} \zeta_N^{xu} \hat{\zeta}(\frac{x}{N},s) = N^{1-s} \zeta(-\frac{u}{N},s)
\end{equation*}

La formule de Hurwitz \cite[(2), Corollary 2 (b)]{knopp-robins} est une équation fonctionnelle reliant $\zeta$ et $\hat{\zeta}$ :
\begin{align}
\label{hurwitz 1} \zeta(x,1-s) & = \frac{\Gamma(s)}{(2\pi)^s} \left(e^{-\frac{i\pi s}{2}} \hat{\zeta}(x,s)+e^{\frac{i\pi s}{2}} \hat{\zeta}(-x,s)\right)\\
\label{hurwitz 2} \hat{\zeta}(x,s) & = \frac{(2\pi)^s}{\Gamma(s) (e^{-i\pi s}-e^{i\pi s})} \left(e^{-i\pi s/2} \zeta(x,1-s) - e^{i\pi s/2} \zeta(-x,1-s)\right).
\end{align}

Rappelons maintenant les résultats concernant les valeurs spéciales de $\zeta$ et $\hat{\zeta}$ aux entiers. Les polynômes de Bernoulli $B_n(x)$ sont définis par
\begin{equation*}
\frac{t e^{xt}}{e^t-1}  = \sum_{n=0}^\infty B_n(x) \frac{t^n}{n!}.
\end{equation*}
On a $B_0(x)=1$, $B_1(x)=x-\frac12$, $B_2(x)=x^2-x+\frac16$.

Pour $x \in \R$, on note $\{x\}=x-\lfloor x \rfloor$ la partie fractionnaire de $x$.

Pour $x \in \R/\Z$ et $n \geq 2$, on a
\begin{equation*}
\zeta(x,1-n) = -\frac{B_n(\{x\})}{n}.
\end{equation*}
On en déduit, pour $x \in \R/\Z$ et $n \geq 2$:
\begin{equation*}
\hat{\zeta}(x,n)+(-1)^n \hat{\zeta}(-x,n) = -\frac{(2i\pi)^n}{n!} B_n(\{x\}).
\end{equation*}
Pour $x \in \R/\Z$, on a
\begin{equation*}
\zeta(x,0) = \begin{cases} \frac12-\{x\} & \textrm{si } x \neq 0\\
-\frac12 & \textrm{si } x=0.
\end{cases}
\end{equation*}
On déduit de la formule de Hurwitz que pour $x \in \R/\Z$, $x \neq 0$, on a
\begin{align}
\label{hatzeta 1} \hat{\zeta}(x,1)-\hat{\zeta}(-x,1) & = 2i\pi \left(\frac12 - \{x\}\right)\\
\label{zeta* 1} \zeta^*(x,1)-\zeta^*(-x,1) & = i\pi \frac{e^{2i\pi x}+1}{e^{2i\pi x}-1}.
\end{align}

Notons les relations suivantes : pour tout $n \geq 1$ et tout $x \in \R/\Z$, avec $x \neq 0$ si $n=1$, on a
\begin{align}
\label{zeta -x} \zeta(-x,1-n)=(-1)^n \zeta(x,1-n) \\
\label{hatzeta -x} \hat{\zeta}(-x,1-n) = (-1)^n \hat{\zeta}(x,1-n).
\end{align}

\begin{definition}
Pour $u \in \Z/N\Z$, définissons les fonctions $\delta_u,\hat{\delta}_u : \Z/N\Z \to \CC$ par
\begin{align*}
\delta_u(n) & = \begin{cases} 1 & \textrm{si } n \equiv u \pmod{N}\\ 
0 & \textrm{si } n \not\equiv u \pmod{N}\end{cases}\\
\hat{\delta}_u(n) & = \sum_{x \in \Z/N\Z} \delta_u(x) \zeta_N^{-xn} = \zeta_N^{-un}.
\end{align*}
\end{definition}

\section{Séries d'Eisenstein-Kronecker}\label{sec EK}

Nous définissons dans cette section les séries d'Eisenstein-Kronecker classiques \cite[\S 1.3]{colmez}, \cite[\S 3]{kato}, \cite[Chap. VII]{schoeneberg}, \cite[Chap. VIII]{weil}.

Pour $k \geq 0$ un entier, $\tau \in \h$, $z,u \in \CC$, on pose
\begin{equation*}
\mathcal{K}_k(s,\tau,z,u) = \frac{\Gamma(s)}{(-2i\pi)^k} \left(\frac{\tau-\overline{\tau}}{2i\pi}\right)^{s-k} \sum_{\substack{\omega \in \Z+\tau\Z \\ \omega \neq -z}} \frac{\overline{\omega+z}^k}{|\omega+z|^{2s}} \exp \left(\frac{2i\pi (\omega \overline{u}-\overline{\omega} u)}{\tau-\overline{\tau}}\right).
\end{equation*}
Cette série converge pour $s \in \CC$, $\Re(s)>1+\frac{k}{2}$ et possède un prolongement méromorphe au plan complexe, holomorphe sur $\CC$ sauf éventuellement des pôles simples en $s=0$ (si $k=0$ et $z \in \Z+\tau\Z$) et en $s=1$ (si $k=0$ et $u \in \Z+\tau\Z$). La fonction $\mathcal{K}_k(s,\tau,z,u)$ est périodique en $u$ de période $\Z+\tau\Z$, et vérifie
\begin{equation*}
\mathcal{K}_k(s,\tau,z+\lambda, u) = \exp\left( \frac{2i\pi (\overline{\lambda}u-\lambda \overline{u})}{\tau-\overline{\tau}}\right) \mathcal{K}_k(s,\tau,z,u) \qquad (\lambda \in \Z+\tau\Z).
\end{equation*}
En particulier la fonction $z \mapsto \mathcal{K}_k(s,\tau,z,0)$ est $(\Z+\tau\Z)$-périodique. La fonction $\mathcal{K}_k$ vérifie l'équation fonctionnelle \cite[Chap. VIII, (32)]{weil}
\begin{equation*}
\mathcal{K}_k(s,\tau,z,u) = \exp \left(\frac{2i\pi(u \overline{z}-\overline{u}z)}{\tau-\overline{\tau}}\right) \mathcal{K}_k(k+1-s,\tau,u,z).
\end{equation*}

Pour $k \geq 1$, $N \geq 1$ un entier, $a,b \in \Z/N\Z$, $\tau \in \h$, on pose
\begin{equation*}
E^{(k)}_{a,b}(\tau) = \mathcal{K}_k\left(k,\tau,\frac{a\tau+b}{N},0\right) \qquad F^{(k)}_{a,b}(\tau) = \mathcal{K}_k\left(k,\tau,0,\frac{a\tau+b}{N}\right).
\end{equation*}

Pour $\tau \in \h$ et $\alpha \in \Q_{>0}$, on pose $q^\alpha = e^{2i\pi \alpha \tau}$.

\begin{lem}
Supposons $k \geq 1$, $k \neq 2$. La fonction $E^{(k)}_{a,b}$ est une série d'Eisenstein de poids $k$ pour le groupe $\Gamma(N)$, et son $q$-développement est donné par
\begin{equation*}
E^{(k)}_{a,b}(\tau) = a_0(E^{(k)}_{a,b}) + \sum_{\substack{m,n \geq 1 \\ m \equiv a (N)}} n^{k-1} \zeta_N^{bn} q^{mn/N} + (-1)^k \sum_{\substack{m,n \geq 1 \\ m \equiv -a(N)}} n^{k-1} \zeta_N^{-bn} q^{mn/N}
\end{equation*}
avec
\begin{equation*}
a_0(E^{(1)}_{a,b}) = \begin{cases} 0 & \textrm{si } a=b=0\\
\frac12 \frac{1+\zeta_N^b}{1-\zeta_N^b} & \textrm{si } a=0 \textrm{ et } b \neq 0\\
\frac12 - \{\frac{a}{N}\} & \textrm{si } a \neq 0
\end{cases}
\end{equation*}
et pour $k \geq 3$,
\begin{equation*}
a_0(E^{(k)}_{a,b}) = \begin{cases} \hat{\zeta}(\frac{b}{N},1-k) & \textrm{si } a=0\\
0 & \textrm{si } a \neq 0.
\end{cases}
\end{equation*}
\end{lem}

\begin{lem}
Supposons $k=2$. La fonction $E^{(2)}_{a,b}$ est de classe $\mathcal{C}^\infty$ sur $\h$. Elle est modulaire de poids $2$ pour le groupe $\Gamma(N)$, et son développement de Fourier est donné par
\begin{equation*}
E^{(2)}_{a,b}(\tau) = a_0(E^{(2)}_{a,b}) + \frac{1}{4\pi \Im(\tau)} +\sum_{\substack{m,n \geq 1\\ m \equiv a (N)}} n \zeta_N^{bn} q^{mn/N} + \sum_{\substack{m,n \geq 1\\ m \equiv -a(N)}} n \zeta_N^{-bn} q^{mn/N}
\end{equation*}
avec
\begin{equation*}
a_0(E^{(2)}_{a,b}) = \begin{cases} \hat{\zeta}(\frac{b}{N},-1) & \textrm{si } a=0,\\
-\frac{1}{12} & \textrm{si } a \neq 0.
\end{cases}
\end{equation*}
En particulier $E^{(2)}_{a,b}-E^{(2)}_{0,0}$ est une série d'Eisenstein de poids $2$ pour le groupe $\Gamma(N)$.
\end{lem}

\begin{lem}\label{lem Fk}
Supposons $k \geq 1$, et $(a,b) \neq (0,0)$ dans le cas $k=2$. La fonction $F^{(k)}_{a,b}$ est une série d'Eisenstein de poids $k$ pour le groupe $\Gamma(N)$, et son $q$-développement est donné par
\begin{equation*}
F^{(k)}_{a,b} = a_0(F^{(k)}_{a,b}) + N^{1-k} \left(\sum_{\substack{m,n \geq 1 \\ n \equiv a(N)}} \zeta_N^{bm} n^{k-1} q^{mn/N} + (-1)^k \sum_{\substack{m,n \geq 1\\ n \equiv -a(N)}} \zeta_N^{-bm} n^{k-1} q^{mn/N}\right).
\end{equation*}
avec
\begin{equation*}
a_0(F^{(1)}_{a,b}) = \begin{cases}0 & \textrm{si } a=b=0\\
\frac12 \frac{1+\zeta_N^b}{1-\zeta_N^b} & \textrm{si } a=0 \textrm{ et } b \neq 0\\
\frac12 - \{\frac{a}{N}\} & \textrm{si } a \neq 0
\end{cases}
\end{equation*}
et pour $k \geq 2$,
\begin{equation*}
a_0(F^{(k)}_{a,b}) = \zeta(\frac{a}{N},1-k) = -\frac{B_k(\{\frac{a}{N}\})}{k}.
\end{equation*}
\end{lem}

\begin{lem}\cite[Lemma 3.7 (1)(iii)]{kato} \label{lem Fk SL2}
Soit $k \geq 1$ un entier. Soit $(a,b) \in (\Z/N\Z)^2$, avec $(a,b) \neq (0,0)$ dans le cas $k=2$. Alors pour tout $g \in \SL_2(\Z)$, on a $F^{(k)}_{a,b} |_k g = F^{(k)}_{(a,b)g}$.
\end{lem}

Nous allons maintenant définir des séries d'Eisenstein dont le $q$-développement est \emph{rationnel}.

\begin{definition}
Pour $k \geq 1$, $a, b \in \Z/N\Z$, on pose
\begin{equation*}
G^{(k)}_{a,b}(\tau) = a_0(G^{(k)}_{a,b}) + \sum_{\substack{m,n \geq 1 \\ m \equiv a, n \equiv b (N)}} m^{k-1} q^{mn} +(-1)^k \sum_{\substack{m,n \geq 1 \\ m \equiv -a, n \equiv -b (N)}} m^{k-1} q^{mn}
\end{equation*}
avec
\begin{equation*}
a_0(G^{(1)}_{a,b}) = \begin{cases} 0 & \textrm{si } a=b=0\\
\frac12-\{\frac{b}{N}\} & \textrm{si } a=0 \textrm{ et } b \neq 0\\
\frac12-\{\frac{a}{N}\} & \textrm{si } a \neq 0 \textrm{ et } b=0\\
0 & \textrm{si } a \neq 0 \textrm{ et } b \neq 0,
\end{cases}
\end{equation*}
et pour $k \geq 2$,
\begin{equation*}
a_0(G^{(k)}_{a,b}) = \begin{cases} - N^{k-1} \frac{B_k(\{\frac{a}{N}\})}{k} & \textrm{si } b=0\\
0 & \textrm{si } b \neq 0.
\end{cases}
\end{equation*}
\end{definition}

\begin{lem}\label{lem Gkab}
Soit $k \geq 1$, avec $a \neq 0$ si $k=2$. Alors la fonction $G^{(k)}_{a,b}(\tau/N)$ est modulaire de poids $k$ pour $\Gamma(N)$.
\end{lem}

\begin{proof}
On vérifie l'identité
\begin{equation*}
F^{(k)}_{a,b}(\tau) = N^{1-k} \sum_{c \in \Z/N\Z} \zeta_N^{bc} G^{(k)}_{a,c}(\tau/N).
\end{equation*}
Par inversion de Fourier, il vient
\begin{equation*}
G^{(k)}_{a,b}(\tau/N) = N^{k-2} \sum_{c \in \Z/N\Z} \zeta_N^{-bc} F^{(k)}_{a,c}(\tau).
\end{equation*}
Le résultat suit alors du lemme \ref{lem Fk}.
\end{proof}

\begin{lem}
Si $G(\tau/N)$ est modulaire de poids $k$ pour le groupe $\Gamma(N)$, alors $G(\tau)$ est modulaire de poids $k$ pour le groupe $\Gamma_1(N^2)$.
\end{lem}

\begin{proof}
Cela résulte de l'inclusion $\begin{pmatrix} N & 0 \\ 0 & 1 \end{pmatrix} \Gamma_1(N^2) \begin{pmatrix} N & 0 \\ 0 & 1 \end{pmatrix}^{-1} \subset \Gamma(N)$.
\end{proof}

On en déduit le lemme suivant.

\begin{lem}\label{lem Gkab 2}
Soit $k \geq 1$, avec $a \neq 0$ si $k=2$. Alors $G^{(k)}_{a,b}$ est une forme modulaire de poids $k$ pour $\Gamma_1(N^2)$.
\end{lem}

Rappelons que l'involution d'Atkin-Lehner $W_N$ sur $M_k(\Gamma_1(N))$ est définie par
\begin{equation*}
(W_N f)(\tau) = i^k N^{-k/2} \tau^{-k} f\left(-\frac{1}{N\tau}\right) \qquad (f \in M_k(\Gamma_1(N))).
\end{equation*}

\begin{definition}\label{def Hkab}
Pour $k \geq 1$, $a, b \in \Z/N\Z$, on pose
\begin{equation*}
H^{(k)}_{a,b}(\tau) = a_0(H^{(k)}_{a,b}) + \sum_{m,n \geq 1} \left(\zeta_N^{-am-bn} + (-1)^k \zeta_N^{am+bn}\right) n^{k-1} q^{mn}
\end{equation*}
avec
\begin{equation*}
a_0(H^{(1)}_{a,b}) = \begin{cases} 0 & \textrm{si } a=b=0\\
-\frac12 \frac{1+\zeta_N^b}{1-\zeta_N^b} & \textrm{si } a=0 \textrm{ et } b \neq 0\\
-\frac12 \frac{1+\zeta_N^a}{1-\zeta_N^a} & \textrm{si } a \neq 0 \textrm{ et } b=0\\
-\frac12 \left(\frac{1+\zeta_N^a}{1-\zeta_N^a} + \frac{1+\zeta_N^b}{1-\zeta_N^b}\right) & \textrm{si } a \neq 0 \textrm{ et } b \neq 0,
\end{cases}
\end{equation*}
et pour $k \geq 2$,
\begin{equation*}
a_0(H^{(k)}_{a,b}) = \hat{\zeta}(-\frac{b}{N},1-k).
\end{equation*}
\end{definition}

Notons l'identité $H^{(k)}_{-a,-b} = (-1)^k H^{(k)}_{a,b}$.

\begin{lem}\label{lem W GH}
Soit $k \geq 1$, $a,b \in \Z/N\Z$. Dans le cas $k=2$, supposons $a \neq 0$. Alors
\begin{equation*}
W_{N^2}(G^{(k)}_{a,b}) = \frac{i^k}{N} H^{(k)}_{a,b}.
\end{equation*}
En particulier $H^{(k)}_{a,b}$ est une forme modulaire de poids $k$ pour $\Gamma_1(N^2)$.
\end{lem}

\begin{proof}
Le cas $k=1$ est traité dans \cite[Lemma 13]{brunault:reg_siegel}. Supposons donc $k \geq 2$, avec $a \neq 0$ si $k=2$. Par définition, on a
\begin{align*}
W_{N^2} (G^{(k)}_{a,b})(\tau) & = i^k N^{-k} \tau^{-k} G^{(k)}_{a,b}(-\frac{1}{N^2 \tau})\\
& = i^k N^{-k} \tau^{-k} N^{k-2} \sum_{c \in \Z/N\Z} \zeta_N^{-bc} F^{(k)}_{a,c}(-\frac{1}{N\tau}).
\end{align*}
D'après le lemme \ref{lem Fk SL2}, il vient
\begin{align*}
W_{N^2} (G^{(k)}_{a,b})(\tau) & = i^k N^{-k} \tau^{-k} N^{k-2} \sum_{c \in \Z/N\Z} \zeta_N^{-bc} (N \tau)^k F^{(k)}_{c,-a}(N\tau)\\
& = i^k N^{k-2} \sum_{c \in \Z/N\Z} \zeta_N^{-bc} F^{(k)}_{c,-a}(N\tau)
\end{align*}
En utilisant le développement de Fourier de $F^{(k)}$ (lemme \ref{lem Fk}), on obtient
\begin{align*}
W_{N^2} (G^{(k)}_{a,b})(\tau) & = i^k N^{k-2} \sum_{c \in \Z/N\Z} \zeta_N^{-bc} a_0(F^{(k)}_{c,-a}) \\
& \quad + i^k N^{k-2} \sum_{c \in \Z/N\Z} \zeta_N^{-bc} N^{1-k} \sum_{m,n \geq 1} \left(\hat{\delta}_a(m) \delta_c(n) + (-1)^k \hat{\delta}_{-a}(m) \delta_{-c}(n)\right) n^{k-1} q^{mn}\\
& = i^k N^{k-2} \sum_{c \in \Z/N\Z} \zeta_N^{-bc} \zeta(\frac{c}{N},1-k) \\
& \quad + \frac{i^k}{N} \sum_{m,n \geq 1} \left(\hat{\delta}_a(m) \hat{\delta}_b(n) + (-1)^k \hat{\delta}_{-a}(m) \hat{\delta}_{-b}(n)\right) n^{k-1} q^{mn}\\
& = \frac{i^k}{N} \left( \hat{\zeta}(-\frac{b}{N},1-k) + \sum_{m,n \geq 1} \left(\hat{\delta}_a(m) \hat{\delta}_b(n) + (-1)^k \hat{\delta}_{-a}(m) \hat{\delta}_{-b}(n)\right) n^{k-1} q^{mn}\right).
\end{align*}
\end{proof}

Pour terminer cette section, nous déterminons la fonction $L$ associée à $H^{(k)}_{a,b}$ (voir \cite[3.10]{kato} pour le cas des séries $E^{(k)}_{a,b}$ et $F^{(k)}_{a,b}$).

Rappelons que si $f=\sum_{n = 0}^\infty a_n(f) q^n$ est une forme modulaire de poids $k \geq 1$ pour $\Gamma_1(N)$, alors la fonction $L$ de $f$ est définie par $L(f,s)=\sum_{n=1}^\infty a_n(f) n^{-s}$ pour $\Re(s)>k$.

\begin{definition}
Pour $f \in M_k(\Gamma_1(N))$, on pose $f^* = f-a_0(f)$.
\end{definition}

La fonction $L$ complétée de $f$ est définie par
\begin{equation*}
\Lambda(f,s):=N^{s/2} (2\pi)^{-s} \Gamma(s) L(f,s) = N^{s/2} \int_0^\infty f^*(iy) y^s \frac{\dd y}{y}.
\end{equation*}
Elle se prolonge en une fonction méromorphe sur $\CC$ qui vérifie l'équation fonctionnelle
\begin{equation*}
\Lambda(f,s) = \Lambda(W_N f, k-s) \qquad (f \in M_k(\Gamma_1(N))).
\end{equation*}
De plus, la fonction $\Lambda(f,s)+\frac{a_0(f)}{s}+\frac{a_0(W_N f)}{k-s}$ est holomorphe sur $\CC$ \cite[Thm 4.3.5]{miyake}.

\begin{definition}\label{def Lambdaf0}
Les valeurs régularisées de $\Lambda(f,s)$ en $s=0$ et $s=k$ sont définies par
\begin{align*}
\Lambda^*(f,0) & := \lim_{s \to 0} \left(\Lambda(f,s)+\frac{a_0(f)}{s}\right)\\
\Lambda^*(f,k) & := \lim_{s \to k} \left(\Lambda(f,s)+\frac{a_0(W_N f)}{k-s}\right).
\end{align*}
\end{definition}

Notons que $\Lambda^*(f,k)=\Lambda^*(W_N f,0)$ pour tout $f \in M_k(\Gamma_1(N))$.

\begin{lem}\label{lem L Hkab}
Soit $k \geq 1$ et $a,b \in \Z/N\Z$. Dans le cas $k=2$, supposons $a \neq 0$. Alors
\begin{equation*}
\Lambda(H^{(k)}_{a,b},s) = N^s (2\pi)^{-s} \Gamma(s) \left(\hat{\zeta}(-\frac{a}{N},s) \hat{\zeta}(-\frac{b}{N},s-k+1) + (-1)^k \hat{\zeta}(\frac{a}{N},s) \hat{\zeta}(\frac{b}{N},s-k+1)\right).
\end{equation*}
\end{lem}

\begin{proof}
Cela résulte des définitions.
\end{proof}

\section{Cycles de Shokurov}\label{sec shokurov}

Soit $N \geq 3$ un entier. Soit $Y(N)$ le courbe modulaire de niveau $N$ définie sur $\Q$, et soit $E$ la courbe elliptique universelle au-dessus de $Y(N)$. Fixons un entier $n \geq 0$. Notons $E^n$ la puissance fibrée $n$-ième de $E$ au-dessus de $Y(N)$. C'est une variété abélienne de dimension relative $n$ sur $Y(N)$. Les points complexes de $E^n$ sont décrits par l'isomorphisme \cite[(3.6)]{deninger:extensions}
\begin{equation*}
E^n(\CC) = (\Z^{2n} \rtimes \SL_2(\Z)) \backslash (\h \times \CC^n \times \GL_2(\Z/N\Z)).
\end{equation*}
On note $\sigma = \begin{pmatrix} 0 & -1 \\ 1 & 0 \end{pmatrix}$. Pour tout entier $0 \leq j \leq n$, on définit le \emph{cycle de Shokurov}
\begin{equation*}
X^j Y^{n-j} \{0,\infty\} = \left\{(iy ; it_1 y,\ldots,it_j y,t_{j+1},\ldots,t_n ; \sigma) : y>0, t_1,\ldots,t_n \in [0,1]\right\}.
\end{equation*}
C'est un $(n+1)$-cycle sur $E^n(\CC)$, muni de l'orientation produit. Il est naturellement fibré au-dessus du symbole modulaire $\{0,\infty\}$. Notons $f_j : X^j Y^{n-j} \{0,\infty\} \to \{0,\infty\}$ cette fibration.

\begin{definition}
Pour $y>0$, on note $\gamma_{y,j}$ la fibre de $f_j$ au-dessus du point $iy$.
\end{definition}

On a donc
\begin{equation*}
\gamma_{y,j} = \left\{(iy,it_1 y,\ldots,it_j y,t_{j+1},\ldots,t_n,\sigma) : t_1,\ldots,t_n \in [0,1]\right\}.
\end{equation*}

\section{Symbole d'Eisenstein}\label{sec Eis}

On note $X(N)$ la compactification de $Y(N)$, et on note $X^\infty = X(N)-Y(N)$ l'ensemble des pointes de $X(N)$, vu comme sous-schéma fermé de $X(N)$. On a une bijection
\begin{align*}
P(\Z/N\Z) \backslash \GL_2(\Z/N\Z) & \xrightarrow{\cong} X^\infty\\
[g] & \mapsto g\infty
\end{align*}
où $P$ est le sous-groupe algébrique de $\SL_2$ formé des matrices de la forme $\pm \begin{pmatrix} 1 & * \\ 0 & 1 \end{pmatrix}$. Pour tout entier $n \geq 0$, on définit
\begin{equation*}
\Q[X^\infty]^{(n)} = \left\{ f : \GL_2(\Z/N\Z) \to \Q : f\left(\begin{pmatrix} * & * \\ 0 & 1 \end{pmatrix} g \right) = f(g) = (-1)^n f(-g)\right\}.
\end{equation*}
Ce groupe est (non canoniquement) isomorphe au groupe des diviseurs sur $X^\infty$.

Fixons un entier $n \geq 0$. On dispose d'une application résidu
\begin{equation*}
\Res^n : H^{n+1}_{\mathcal{M}}(E^n,\Q(n+1)) \to \Q[X^\infty]^{(n)}.
\end{equation*}
généralisant l'application diviseur pour $n=0$. Notons $V_0$ le sous-groupe de $\Q[X^\infty]^{(0)}$ formé des diviseurs de degré $0$, et notons $V_n = \Q[X^\infty]^{(n)}$ pour $n \geq 1$. L'image de l'application $\Res^n$ est égale à $V_n$ (généralisation du théorème de Manin-Drinfeld). Le \emph{symbole d'Eisenstein}, construit par Beilinson, est une application canonique
\begin{equation*}
\BB^n : V_n \to H^{n+1}_{\mathcal{M}}(E^n,\Q(n+1))
\end{equation*}
telle que $\Res^n \circ \BB^n = \id_{V_n}$.

Notons $\omega^n : \Q[(\Z/N\Z)^2] \to \Q[X^\infty]^{(n)}$ l'application horosphérique, définie par
\begin{equation*}
\omega^n(\phi)(g) = \sum_{v \in (\Z/N\Z)^2} \phi(g^{-1}v) B_{n+2}(\{\frac{v_2}{N}\})
\end{equation*}
Dans le cas $n=0$, l'application $\omega^0$ induit une surjection de $\Q[(\Z/N\Z)^2 \backslash \{0\}]$ dans $V_0$ ; dans le cas $n \geq 1$, l'application $\omega^n$ est surjective \cite[7.5]{schappacher-scholl:eis} (la preuve donnée dans cet article marche également pour $n=0$).

\begin{definition}
Soit $u \in (\Z/N\Z)^2$. Dans le cas $n=0$, on suppose $u \neq 0$. On définit
\begin{equation}
\Eis^n(u) = \BB^n \circ \omega^n (\phi_u) \in H^{n+1}_{\mathcal{M}}(E^n,\Q(n+1))
\end{equation}
où $\phi_u$ est la fonction caractéristique de $\{u\}$. 
\end{definition}

Notons que $\Eis^0(u)$ n'est autre que l'unité de Siegel $g_u \otimes \frac{2}{N}$ \cite[\S 1]{kato}.

Par définition, on a $\Res^n(\Eis^n(u)) =\omega^n(\phi_u)$ ; l'application horosphérique donne donc les résidus des symboles d'Eisenstein.

Le groupe $G=\GL_2(\Z/N\Z)$ agit à gauche sur $E^n$. Sur les points complexes, cette action est induite par $\gamma \cdot (\tau,z,g)=(\tau,z,g{}^{t}\gamma)$. On en déduit une action à droite de $G$ sur $H^{n+1}_{\mathcal{M}}(E^n,\Q(n+1))$. Les applications $\Res^n$ et $\BB^n$ étant $G$-équivariantes, on a le lemme suivant.

\begin{lem}
Pour tout $g \in G$, et pour tout $u \in (\Z/N\Z)^2$ (avec $u \neq 0$ si $n=0$), on a
\begin{equation}
g^* \Eis^n(u) = \Eis^n ({}^{t}gu).
\end{equation}
\end{lem}

Dans la suite de cette section, nous rappelons la description explicite de la réalisation du symbole d'Eisenstein en cohomologie de Deligne \cite{beilinson:2,deninger:extensions,deninger-scholl}.

Commençons par quelques rappels sur la cohomologie de Deligne $H^i_{\mathcal{D}}(X/\R,\R(j))$ associée à un schéma $X$ lisse et quasi-projectif sur $\R$. Notons $\mathcal{S}^\cdot(X,\R(n))$ le complexe des formes différentielles $\mathcal{C}^\infty$ sur $X(\CC)$ à valeurs dans $\R(n)=(2i\pi)^n \R$ vérifiant $c^* \omega = (-1)^n \omega$, où $c$ désigne la conjugaison complexe sur $X(\CC)$. Par la résolution des singularités, il existe une compactification lisse $\overline{X}$ de $X$ telle que $X^\infty := \overline{X}-X$ soit un diviseur à croisements normaux. Pour un entier $m \geq 0$, notons $\Omega^m_{\overline{X}} \langle X^\infty \rangle$ le $\CC$-espace vectoriel des $m$-formes holomorphes sur $X(\CC)$ à singularités logarithmiques le long de $X^\infty(\CC)$. Pour toute forme différentielle $\alpha$ et tout entier $n \in \Z$, notons $\pi_n(\alpha) = \frac12(\alpha+(-1)^n \overline{\alpha})$.

\begin{pro}\cite[(2.5.1)]{deninger-scholl}\label{pro description HnXn}
Pour tout entier $n \geq 0$, on a un isomorphisme
\begin{equation*}
H^{n+1}_{\mathcal{D}}(X/\R,\R(n+1)) \cong \frac{\{\phi \in \mathcal{S}^n(X,\R(n)) | d\phi = \pi_n(\omega) \textrm{ avec } \omega \in \Omega^{n+1}_{\overline{X}}\langle X^\infty \rangle\}}{d\mathcal{S}^{n-1}(X,\R(n))}.
\end{equation*}
\end{pro}

Considérons maintenant le régulateur de Beilinson
\begin{equation*}
r_{\BB} : H^{n+1}_{\mathcal{M}}(E^n,\Q(n+1)) \to H^{n+1}_{\mathcal{D}}(E^n/\R,\R(n+1)).
\end{equation*}
D'après la proposition \ref{pro description HnXn}, les éléments de $H^{n+1}_{\mathcal{D}}(E^n/\R,\R(n+1))$ sont représentés par des $n$-formes différentielles sur $E^n(\CC)$. Nous noterons $(\tau;z_1,\ldots,z_n;g)$ les coordonnées naturelles sur $E^n(\CC)$, avec $\tau \in \h$, $z_i \in \CC$, $g \in \GL_2(\Z/N\Z)$, et nous poserons $q=e^{2i\pi \tau}$. Pour tous entiers $a,b \geq 0$ tels que $a+b=n$, définissons la $n$-forme différentielle sur $\CC^n$
\begin{equation*}
\psi_{a,b} = \frac{1}{n!} \sum_{\sigma \in \mathfrak{S}_n} \varepsilon(\sigma) d\overline{z}_{\sigma(1)} \wedge \cdots \wedge d\overline{z}_{\sigma(b)} \wedge dz_{\sigma(b+1)} \wedge \cdots \wedge dz_{\sigma(n)}.
\end{equation*}

D'après \cite[(3.12), (3.28)]{deninger:extensions} et \cite[Remark after Lemma 7.1]{huber-kings}, on a la proposition suivante.

\begin{pro}\label{pro EisD}
Soit $u \in (\Z/N\Z)^2$, avec $u \neq 0$ dans le cas $n=0$. L'élément $r_{\BB}(\Eis^n(u))$ est représenté par une $n$-forme différentielle $\Eis^n_{\mathcal{D}}(u)$ vérifiant
\begin{equation*}
\Eis^n_{\mathcal{D}}(u) = -\frac{n! (n+2)}{2i\pi N} \frac{\tau-\overline{\tau}}{2} \sum_{a=0}^n \left(\sideset{}{'}\sum_{(c,d) \in \Z^2} \frac{\zeta_N^{c(gu)_1+d(gu)_2}}{(c\tau+d)^{a+1} (c\overline{\tau}+d)^{n+1-a}} \right) \psi_{a,n-a} \mod{d\tau,d\overline{\tau}}.
\end{equation*}
De plus, on a
\begin{equation*}
\dd \Eis^n_{\mathcal{D}}(u) = \pi_n(\Eis^n_{\hol}(u))
\end{equation*}
où $\Eis^n_{\hol}(u)$ est la série d'Eisenstein holomorphe de poids $n+2$ définie par
\begin{equation*}
\Eis^n_{\hol} (u) = -\frac{(n+2)!}{(2i\pi)^2 N} \sideset{}{'}\sum_{(c,d) \in \Z^2} \frac{\zeta_N^{c(gu)_1+d(gu)_2}}{(c\tau+d)^{n+2}} \frac{dq}{q} \wedge dz_1 \wedge \cdots \wedge dz_n.
\end{equation*}
\end{pro}

\begin{remarque}
Pour $n=0$, les séries intervenant dans $\Eis^0_{\mathcal{D}}(u)$ et $\Eis^0_{\hol}(u)$ ne convergent pas absolument et on a recours à la sommation d'Eisenstein ou de Kronecker pour leur donner un sens \cite[Chap. VIII, \S 9]{weil}.
\end{remarque}

\section{Éléments de Deninger-Scholl}\label{deninger-scholl-gealy}

Dans cette section, nous rappelons la définition des éléments de Deninger-Scholl et calculons explicitement leurs réalisations en cohomologie de Deligne.

Soit $k \geq 0$ un entier. On choisit une décomposition $k=k_1+k_2$ avec $k_1,k_2 \geq 0$. Considérons les projections canoniques $p_1 : E^{k_1+k_2} \to E^{k_1}$ et $p_2 : E^{k_1+k_2} \to E^{k_2}$. Les éléments suivants, définis par Gealy dans \cite{gealy:thesis}, sont une généralisation des éléments de Beilinson-Kato et Deninger-Scholl.

\begin{definition}\label{def zkj}
Soient $u_1,u_2 \in (\Z/N\Z)^2$. Dans le cas où $k_i=0$, on suppose $u_i \neq 0$. On pose
\begin{equation}
\Eis^{k_1,k_2}(u_1,u_2) = p_1^* \Eis^{k_1}(u_1) \cup p_2^* \Eis^{k_2}(u_2) \in H^{k+2}_{\mathcal{M}}(E^k,\Q(k+2)).
\end{equation}
\end{definition}

Dans le cas $k=0$, on retrouve les éléments de Beilinson-Kato dans le $K_2$ de la courbe modulaire $Y(N)$ \cite{kato}. Notons également la relation $\Eis^{0,k}(u_2,u_1)=(-1)^{k+1} \Eis^{k,0}(u_1,u_2)$, qui découle du caractère commutatif gradué du cup-produit \cite[(1.3) Thm (2)]{deninger-scholl}.

Considérons le régulateur de Beilinson
\begin{equation*}
r_{\BB} : H^{k+2}_{\mathcal{M}}(E^k,\Q(k+2)) \to H^{k+2}_{\mathcal{D}}(E^k/\R,\R(k+2)).
\end{equation*}

Les éléments de $H^{k+2}_{\mathcal{D}}(E^k/\R,\R(k+2))$ sont représentés par des $(k+1)$-formes différentielles fermées sur $E^k(\CC)$. En effet, comme $E^k$ est de dimension $k+1$, il n'y a pas de $(k+2)$-forme holomorphe sur $E^k(\CC)$, et la proposition \ref{pro description HnXn} entraîne un isomorphisme
\begin{equation*}
H^{k+2}_{\mathcal{D}}(E^k/\R,\R(k+2)) \cong H^{k+1}(\mathcal{S}^\cdot(E^k/\R,\R(k+1))).
\end{equation*}

\begin{lem}
L'élément $r_{\BB}(\Eis^{k_1,k_2}(u_1,u_2))$ est représenté par la forme différentielle fermée
\begin{equation}
\begin{split}
\Eis_{\mathcal{D}}^{k_1,k_2}(u_1,u_2) & = p_1^* \Eis^{k_1}_{\mathcal{D}}(u_1) \wedge \pi_{k_2+1} (p_2^* \Eis^{k_2}_{\hol}(u_2)) \\
& \qquad + (-1)^{k_1+1} \pi_{k_1+1} (p_1^* \Eis^{k_1}_{\hol}(u_1)) \wedge p_2^* \Eis^{k_2}_{\mathcal{D}}(u_2).
\end{split}
\end{equation}
\end{lem}

\begin{proof}
Cela résulte de la proposition \ref{pro EisD} et de la formule pour le cup-produit en cohomologie de Deligne \cite[(2.5)]{deninger-scholl}.
\end{proof}

\section{La méthode de Rogers-Zudilin}\label{sec RZ}

Rogers et Zudilin \cite{rogers-zudilin} ont introduit une nouvelle méthode permettant de calculer certaines mesures de Mahler en termes de valeurs de fonctions $L$. Cette méthode est basée sur un changement de variables astucieux dans une intégrale le long du symbole modulaire $\{0,\infty\}$. Dans cette section, nous expliquons cette méthode dans un cadre assez général (voir également l'interprétation proposée dans \cite{diamantis-neururer-stromberg}). L'identité obtenue (Proposition \ref{pro RZ trick}) est la clé de notre calcul du régulateur.

Introduisons, pour des fonctions $\alpha,\beta : \Z/N\Z \to \CC$ et des nombres complexes $t,u \in \CC$, la série suivante
\begin{equation}
S^{t,u}_{\alpha,\beta}(\tau) = \sum_{m \geq 1} \sum_{n \geq 1} \alpha(m) \beta(n) m^{t} n^{u} q^{mn/N} \qquad (q=e^{2i\pi \tau}).
\end{equation}
Cette série définit une fonction holomorphe sur $\h$.

\begin{lem}\label{lem Skl}
La fonction $S^{t,u}_{\alpha,\beta}(iy)$ décroît exponentiellement lorsque $y \to +\infty$, et croît au plus polynomialement en $\frac{1}{y}$ lorsque $y \to 0$.
\end{lem}

\begin{proof}
La décroissance exponentielle lorsque $y \to +\infty$ résulte de la définition. Pour établir la croissance polynomiale lorsque $y \to 0$, il suffit de le faire pour la série
\begin{equation*}
\sum_{m,n \geq 1} (mn)^k e^{-2\pi mn y}.
\end{equation*}
où $k$ est un entier naturel. Puisque $(mn)^k \leq \frac12 (m^{2k}+n^{2k})$, on est ramené au cas de la série
\begin{equation}\label{lem Skl serie}
\sum_{m,n \geq 1} m^\ell e^{-2\pi mn y}
\end{equation}
où $\ell$ est un entier naturel, que l'on peut supposer impair. Il s'agit alors, au terme constant près, d'une série d'Eisenstein holomorphe $E$ de poids $\ell+1$ pour le groupe $\SL_2(\Z)$. Puisque $E(-1/\tau)=\tau^{\ell+1} E(\tau)$, la série (\ref{lem Skl serie}) est $\ll y^{-\ell-1}$ lorsque $y \to 0$.
\end{proof}

\begin{lem}\label{lem mellin Skl}
Pour $s \in \CC$, $\Re(s) \gg 0$, la fonction $y \mapsto S^{t,u}_{\alpha,\beta}(iy) y^{s-1}$ est intégrable sur $]0,+\infty[$, et on a
\begin{equation*}
\int_0^\infty S^{t,u}_{\alpha,\beta}(iy) y^s \frac{\dd y}{y} = (\frac{2\pi}{N})^{-s} \Gamma(s) L(\alpha,s-t) L(\beta,s-u)
\end{equation*}
Pour $s \in \CC$, $\Re(s) \ll 0$, la fonction $y \mapsto S^{t,u}_{\alpha,\beta}(\frac{i}{y}) y^{s-1}$ est intégrable sur $]0,+\infty[$, et on a
\begin{equation*}
\int_0^\infty S^{t,u}_{\alpha,\beta}(\frac{i}{y}) y^s \frac{\dd y}{y} = (\frac{2\pi}{N})^s \Gamma(-s) L(\alpha,-s-t) L(\beta,-s-u)
\end{equation*}
\end{lem}

\begin{proof}
D'après le lemme \ref{lem Skl}, la première intégrale converge pour $\Re(s) \gg 0$. On intervertit somme et intégrale, et on effectue le changement de variable $y'=\frac{2\pi mny}{N}$ ; un calcul simple mène alors au résultat. La seconde assertion résulte de la première après le changement de variable $y \mapsto 1/y$.
\end{proof}

\begin{lem}\label{lem int bessel}
Pour tout $k \in \Z$, il existe des constantes $C>0$ et $\mu,\nu \geq 0$ telles que pour tous réels $a,b \geq 1$, on ait
\begin{equation*}
\int_0^\infty e^{-ay-b/y} y^k \frac{\dd y}{y} \leq C a^\mu b^\nu e^{-\sqrt{ab}}.
\end{equation*}
\end{lem}

\begin{proof}
La fonction $f : y \mapsto ay+b/y$ admet un minimum en $y=\sqrt{b/a}$, où elle prend la valeur $2\sqrt{ab}$. On découpe l'intégrale en deux : $\int_0^\infty = \int_0^{\sqrt{b/a}} + \int_{\sqrt{b/a}}^\infty$. En faisant le changement de variables $y'=b/(ay)$, il suffit de majorer l'intégrale sur $[\sqrt{b/a},+\infty[$. On a
\begin{equation*}
\int_{\sqrt{b/a}}^\infty e^{-ay-b/y} y^k \frac{\dd y}{y} \leq \int_{\sqrt{b/a}}^\infty e^{-ay} y^k \frac{\dd y}{y} \leq a^{-k} \int_{\sqrt{ab}}^\infty e^{-y} y^k \frac{\dd y}{y}.
\end{equation*}
Si $k \leq 1$, on a $y^{k-1} \leq (ab)^{(k-1)/2}$ d'où la majoration
\begin{equation*}
\int_{\sqrt{b/a}}^\infty e^{-ay-b/y} y^k \frac{\dd y}{y} \leq a^{-k} (ab)^{(k-1)/2} e^{-\sqrt{ab}}.
\end{equation*}
Si $k > 1$, une intégration par parties donne
\begin{equation*}
\int_{\sqrt{ab}}^\infty e^{-y} y^{k-1} \dd y = (ab)^{(k-1)/2} e^{-\sqrt{ab}} + (k-1) \int_{\sqrt{ab}}^\infty e^{-y} y^{k-2} \dd y.
\end{equation*}
Une récurrence sur $k$ donne alors la majoration voulue.
\end{proof}

Nous pouvons maintenant énoncer la proposition-clé à la base de la méthode de Rogers-Zudilin.

\begin{pro}\label{pro RZ trick}
Soient $t_1,u_1,t_2,u_2,s \in \CC$ des nombres complexes, et soient $\alpha_1,\beta_1,\alpha_2,\beta_2 : \Z/N\Z \to \CC$ des fonctions. On a
\begin{equation}\label{eq RZ trick}
\int_0^\infty S^{t_1,u_1}_{\alpha_1,\beta_1}\left(\frac{i}{y}\right) S^{t_2,u_2}_{\alpha_2,\beta_2}(iy) y^s \frac{\dd y}{y} = \int_0^\infty S^{t_1+s,t_2}_{\alpha_1,\alpha_2}(iy) S^{u_1,u_2-s}_{\beta_1,\beta_2}\left(\frac{i}{y}\right) y^s \frac{\dd y}{y}.
\end{equation}
\end{pro}

\begin{proof}
Remarquons que ces intégrales convergent absolument d'après le lemme \ref{lem Skl}. Par définition, on a
\begin{equation*}
S^{t_1,u_1}_{\alpha_1,\beta_1}\left(\frac{i}{y}\right) S^{t_2,u_2}_{\alpha_2,\beta_2}(iy) = \sum_{\substack{m_1,n_1 \geq 1 \\ m_2,n_2 \geq 1}} \alpha_1(m_1) \beta_1(n_1) \alpha_2(m_2) \beta_2(n_2) m_1^{t_1} n_1^{u_1} m_2^{t_2} n_2^{u_2} e^{-\frac{2\pi}{N}\left(\frac{m_1 n_1}{y}+m_2 n_2 y\right)}.
\end{equation*}
Grâce au lemme \ref{lem int bessel}, on peut intervertir l'intégrale et la sommation dans le membre de gauche de (\ref{eq RZ trick}). En effectuant le changement de variables $y' = (n_2/m_1)y$, il vient
\begin{equation*}
\int_0^\infty \exp \left(-\frac{2\pi}{N} \left(\frac{m_1 n_1}{y}+m_2 n_2 y\right)\right) y^s \frac{\dd y}{y} = \left(\frac{m_1}{n_2}\right)^s \int_0^\infty \exp \left(-\frac{2\pi}{N} \left(m_1 m_2 y + \frac{n_1 n_2}{y}\right)\right) y^s \frac{\dd y}{y}.
\end{equation*}
En échangeant à nouveau intégrale et sommation, on obtient le résultat.
\end{proof}

\section{Quelques développements de Fourier}\label{sec fourier}

Dans cette section, nous calculons le développement de Fourier de la réalisation du symbole d'Eisenstein. Commençons par rappeler le développement de Fourier de $\Eis^n_{\mathrm{hol}}(u)$, qui a en fait déjà été déterminé dans la section \ref{sec EK}.

\begin{pro}\label{pro Enu}
Soit $n \geq 0$ un entier, et soit $u \in (\Z/N\Z)^2$ (avec $u \neq 0$ si $n=0$). Alors
\begin{equation*}
\Eis^n_{\mathrm{hol}}(u) = (-1)^{n+1} \frac{n+2}{N} (2i\pi)^n F^{(n+2)}_{\sigma gu}(\tau) \frac{\dd q}{q} \wedge \dd z_1 \wedge \cdots \wedge \dd z_n.
\end{equation*}
\end{pro}

\begin{proof}
Par définition, on a
\begin{align*}
\sideset{}{'}\sum_{(c,d) \in \Z^2} \frac{\zeta_N^{c(gu)_1+d(gu)_2}}{(c\tau+d)^{n+2}} & = \frac{(-2i\pi)^{n+2}}{(n+1)!} \mathcal{K}_{n+2}\left(n+2,\tau,0,\frac{-(gu)_2 \tau+(gu)_1}{N}\right) \\
& = \frac{(-2i\pi)^{n+2}}{(n+1)!} F^{(n+2)}_{-(gu)_2,(gu)_1}(\tau) \\
& = \frac{(-2i\pi)^{n+2}}{(n+1)!} F^{(n+2)}_{\sigma gu}(\tau).
\end{align*}
On en déduit le résultat.
\end{proof}

La forme différentielle $\Eis^n_{\mathcal{D}}(u)$ étant invariante par $\tau \mapsto \tau+N$, elle possède un développement de Fourier en la variable $q^{1/N}=e^{2i\pi \tau/N}$.

Introduisons, pour des entiers $a,b \geq 0$, et $u=(u_1,u_2) \in (\Z/N\Z)^2$, les séries d'Eisenstein analytiques-réelles suivantes
\begin{equation*}
E^{a,b}_u(\tau) = \sideset{}{'}\sum_{(m,n) \equiv (u_1,u_2)(N)} \frac{1}{(m\tau+n)^{a+1} (m\overline{\tau}+n)^{b+1}} \qquad (\tau \in \h).
\end{equation*}
\begin{equation*}
F^{a,b}_u(\tau) = \sideset{}{'}\sum_{(m,n) \in \Z^2} \frac{\zeta_N^{mu_1+nu_2}}{(m\tau+n)^{a+1} (m\overline{\tau}+n)^{b+1}} \qquad (\tau \in \h).
\end{equation*}
On a donc
\begin{equation}\label{Fab Eab}
F^{a,b}_u = \sum_{x,y \in \Z/N\Z} \zeta_N^{xu_1+yu_2} E^{a,b}_{(x,y)}.
\end{equation}
On a alors
\begin{equation}\label{formule EisnD Fgu}
\Eis^n_{\mathcal{D}}(u) = -\frac{n! (n+2)}{2i\pi N} \frac{\tau-\overline{\tau}}{2} \sum_{a=0}^n F^{a,n-a}_{gu}(\tau) \psi_{a,n-a} \mod{\dd\tau,\dd\overline{\tau}}.
\end{equation}

Pour des entiers $k,\ell \in \Z$ et des fonctions $\alpha,\beta : \Z/N\Z \to \CC$, nous poserons également
\begin{equation*}
\overline{S}^{k,\ell}_{\alpha,\beta}(\tau) = \overline{S^{k,\ell}_{\alpha,\beta}(\tau)} = \sum_{m \geq 1} \sum_{n \geq 1} \overline{\alpha}(m) \overline{\beta}(n) m^k n^\ell \overline{q}^{mn/N}.
\end{equation*}

\begin{pro}\label{fourier Eab}
Pour $a,b \geq 0$ et $u_1,u_2 \in \Z/N\Z$, on a
\begin{equation}
\begin{split}
E^{a,b}_{(u_1,u_2)}(\tau) & = \delta_0(u_1) N^{-a-b-2} \left(\zeta\left(\frac{u_2}{N},a+b+2\right)+(-1)^{a+b} \zeta\left(-\frac{u_2}{N},a+b+2\right)\right)\\
& \quad + (-1)^b \frac{2i\pi}{N^{a+b+2}} \begin{pmatrix} a+b \\ a \end{pmatrix} \left(\zeta\left(\frac{u_1}{N},a+b+1\right)+(-1)^{a+b} \zeta\left(-\frac{u_1}{N},a+b+1\right)\right) (\tau-\overline{\tau})^{-a-b-1}\\
& \quad + \frac{(-1)^{b+1}}{b!} \sum_{j=0}^{a} \frac{(a+b-j)!}{j! (a-j)!} \left(-\frac{2i\pi}{N}\right)^{j+1} (\tau-\overline{\tau})^{-a-b-1+j} \bigl(S^{j-a-b-1,j}_{\delta_{u_1},\hat{\delta}_{-u_2}}(\tau)+(-1)^{a+b} S^{j-a-b-1,j}_{\delta_{-u_1},\hat{\delta}_{u_2}}(\tau) \bigr)\\
& \quad + \frac{(-1)^{b+1}}{a!} \sum_{j=0}^{b} \frac{(a+b-j)!}{j! (b-j)!} \left(-\frac{2i\pi}{N}\right)^{j+1} (\tau-\overline{\tau})^{-a-b-1+j} \bigl(\overline{S}^{j-a-b-1,j}_{\delta_{u_1},\hat{\delta}_{-u_2}}(\tau)+(-1)^{a+b} \overline{S}^{j-a-b-1,j}_{\delta_{-u_1},\hat{\delta}_{u_2}}(\tau) \bigr).
\end{split}
\end{equation}
\end{pro}

\begin{proof}
On suit la méthode classique \cite[Chap III, \S 2]{schoeneberg} pour déterminer le développement de Fourier des séries d'Eisenstein. Pour $a,b \geq 0$, $\tau \in \CC-\R$ et $x \in \R$, posons
\begin{equation*}
\phi_{a,b,\tau}(x) = \sum_{n \in \Z} \frac{1}{(\tau+n+x)^{a+1} (\overline{\tau}+n+x)^{b+1}}
\end{equation*}
Cette série converge et définit une fonction $1$-périodique de $x$. Notons
\begin{equation*}
\phi_{a,b,\tau}(x)=\sum_{r \in \Z} c_r(\phi_{a,b,\tau}) e^{2i\pi rx}
\end{equation*}
son développement de Fourier. On a
\begin{align*}
c_r(\phi_{a,b,\tau}) & = \int_0^1 \phi_{a,b,\tau}(x) e^{-2i\pi rx} dx \\
& = \int_0^1 \sum_{n \in \Z} \frac{ e^{-2i\pi rx} }{(\tau+n+x)^{a+1} (\overline{\tau}+n+x)^{b+1}} dx \\
& = \int_{-\infty}^{\infty} \frac{ e^{-2i\pi rz} }{(\tau+z)^{a+1} (\overline{\tau}+z)^{b+1}} dz.
\end{align*}
Il s'agit de l'intégrale d'une fonction méromorphe ayant pour seuls pôles $z=-\tau$ et $z=-\overline{\tau}$.

\emph{Premier cas : $\tau \in \h$, $r \geq 1$.} On déforme le chemin d'intégration vers le bas, ce qui donne
\begin{align*}
\int_{-\infty}^{\infty} \frac{ e^{-2i\pi rz} }{(\tau+z)^{a+1} (\overline{\tau}+z)^{b+1}} dz & = - 2i\pi \Res_{z=-\tau} \left(\frac{e^{-2i\pi rz}}{(z+\tau)^{a+1} (z+\overline{\tau})^{b+1}}\right) \\
& = -2i\pi \frac{f^{(a)}(-\tau)}{a!}
\end{align*}
où $f$ est la fonction définie par $f(z)=(z+\overline{\tau})^{-b-1} e^{-2i\pi rz}$. Par la formule de Leibniz
\begin{align*}
f^{(a)}(z) & = \sum_{j=0}^a \begin{pmatrix} a \\ j \end{pmatrix} (-b-1) \cdots (-b-j) (z+\overline{\tau})^{-b-1-j} (-2i\pi r)^{a-j} e^{-2i\pi rz} \\
& = \sum_{j=0}^a (-1)^j \begin{pmatrix} a \\ j \end{pmatrix} \frac{(b+j)!}{b!} (z+\overline{\tau})^{-b-1-j} (-2i\pi r)^{a-j} e^{-2i\pi rz}
\end{align*}
d'où l'on déduit
\begin{align*}
c_r(\phi_{a,b,\tau}) & = -2i\pi  \sum_{j=0}^a (-1)^j \frac{(b+j)!}{b! j! (a-j)!} (-\tau+\overline{\tau})^{-b-1-j} (-2i\pi r)^{a-j} e^{2i\pi r\tau} \\
& = (-1)^{a+b} 2i\pi  \sum_{j=0}^a (-1)^j \frac{(b+j)!}{b! j! (a-j)!} (\tau-\overline{\tau})^{-b-1-j} (2i\pi r)^{a-j} e^{2i\pi r\tau}.
\end{align*}
En effectuant le changement d'indices $j \to a-j$, il vient
\begin{align*}
c_r(\phi_{a,b,\tau}) & = (-1)^{a+b} 2i\pi  \sum_{j=0}^a (-1)^{a-j} \frac{(b+a-j)!}{b! j! (a-j)!} (\tau-\overline{\tau})^{-b-1-a+j} (2i\pi r)^j e^{2i\pi r\tau}\\
& = \frac{(-1)^b}{b!} 2i\pi  \sum_{j=0}^a (-1)^j \frac{(a+b-j)!}{j! (a-j)!} (\tau-\overline{\tau})^{-a-b-1+j} (2i\pi r)^j e^{2i\pi r\tau}
\end{align*}

\emph{Deuxième cas : $\tau \in \h$, $r \leq -1$.} Par le changement de variables $z \to -z$, on a
\begin{align*}
\int_{-\infty}^{\infty} \frac{ e^{-2i\pi rz} }{(\tau+z)^{a+1} (\overline{\tau}+z)^{b+1}} dz & = \int_{\infty}^{-\infty} \frac{ e^{2i\pi rz} }{(\tau-z)^{a+1} (\overline{\tau}-z)^{b+1}} d(-z) \\
& = (-1)^{a+b} \int_{-\infty}^{\infty} \frac{ e^{2i\pi rz} }{(-\tau+z)^{a+1} (-\overline{\tau}+z)^{b+1}} dz\\
& = (-1)^{a+b} c_{-r}(\phi_{b,a,-\overline{\tau}})
\end{align*}
et donc
\begin{align*}
c_r(\phi_{a,b,\tau}) & = \frac{(-1)^b}{a!} 2i\pi  \sum_{j=0}^b (-1)^j \frac{(a+b-j)!}{j! (b-j)!} (\tau-\overline{\tau})^{-a-b-1+j} (-2i\pi r)^j e^{2i\pi r\overline{\tau}}\\
& = \frac{(-1)^b}{a!} 2i\pi  \sum_{j=0}^b \frac{(a+b-j)!}{j! (b-j)!} (\tau-\overline{\tau})^{-a-b-1+j} (2i\pi r)^j e^{2i\pi r\overline{\tau}}
\end{align*}

\emph{Troisième cas : $\tau \in \h$, $r =0$.} En déformant le chemin d'intégration vers le bas, on a encore
\begin{align*}
\int_{-\infty}^{\infty} \frac{ 1 }{(\tau+z)^{a+1} (\overline{\tau}+z)^{b+1}} dz & = - 2i\pi \Res_{z=-\tau} \left(\frac{1}{(z+\tau)^{a+1} (z+\overline{\tau})^{b+1}}\right) \\
& = -2i\pi \frac{f^{(a)}(-\tau)}{a!}
\end{align*}
où $f$ est la fonction définie par $f(z)=(z+\overline{\tau})^{-b-1}$. Par la formule de Leibniz
\begin{align*}
f^{(a)}(z) & = (-b-1) \cdots (-b-a) (z+\overline{\tau})^{-b-1-a} \\
& = (-1)^a \frac{(b+a)!}{b!} (z+\overline{\tau})^{-b-1-a}
\end{align*}
d'où l'on déduit
\begin{equation*}
c_0(\phi_{a,b,\tau}) = (-1)^b 2i\pi \begin{pmatrix} a+b \\ a \end{pmatrix} (\tau-\overline{\tau})^{-a-b-1}.
\end{equation*}

\emph{Quatrième cas : $\tau \in -\h$.} On a $c_r(\phi_{a,b,\tau}) = (-1)^{a+b} c_{-r}(\phi_{a,b,-\tau})$.

Exprimons maintenant $E^{a,b}_u$ en termes de $\phi_{a,b,\tau}$. On a

\begin{align*}
E^{a,b}_{(u_1,u_2)}(\tau) & = \delta_{u_1=0} \sideset{}{'}\sum_{n \equiv u_2(N)} \frac{1}{n^{a+b+2}} + \sum_{\substack{m \equiv u_1 (N)\\ m \neq 0}} \sum_{n \in \Z} \frac{1}{(m\tau+u_2+Nn)^{a+1} (m\overline{\tau}+u_2+Nn)^{b+1}}\\
& = \delta_{u_1=0} \left(\sum_{\substack{n \geq 1 \\ n \equiv u_2(N)}} \frac{1}{n^{a+b+2}} +  \sum_{\substack{n \leq -1 \\ n \equiv u_2(N)}} \frac{1}{n^{a+b+2}} \right) + N^{-a-b-2} \sum_{\substack{m \equiv u_1 (N)\\ m \neq 0}} \sum_{n \in \Z} \frac{1}{(\frac{m\tau+u_2}{N}+n)^{a+1} (\frac{m\overline{\tau}+u_2}{N}+n)^{b+1}}\\
& = \delta_{u_1=0} \left(\sum_{\substack{n \geq 1 \\ n \equiv u_2(N)}} \frac{1}{n^{a+b+2}} + (-1)^{a+b} \sum_{\substack{n \geq 1 \\ n \equiv -u_2(N)}} \frac{1}{n^{a+b+2}} \right) + N^{-a-b-2} \sum_{\substack{m \equiv u_1 (N) \\ m \neq 0}} \phi_{a,b,\frac{m\tau + u_2}{N}}(0) \\
& = \delta_{u_1=0} N^{-a-b-2} \left(\sum_{\substack{x >0 \\ x \equiv \frac{u_2}{N} (1)}} \frac{1}{x^{a+b+2}} + (-1)^{a+b} \sum_{\substack{x >0 \\ x \equiv -\frac{u_2}{N} (1)}} \frac{1}{x^{a+b+2}} \right) + N^{-a-b-2} \sum_{\substack{m \equiv u_1 (N) \\ m \neq 0}} \sum_{r \in \Z} c_r(\phi_{a,b,\frac{m\tau + u_2}{N}}) \\
& = \delta_{u_1=0} N^{-a-b-2} \left(\zeta\left(\frac{u_2}{N},a+b+2\right) + (-1)^{a+b} \zeta\left(-\frac{u_2}{N},a+b+2\right) \right) \\
& \qquad + N^{-a-b-2} \left( \sum_{\substack{m \geq 1 \\ m \equiv u_1 (N)}} \sum_{r \in \Z} c_r(\phi_{a,b,\frac{m\tau + u_2}{N}}) + (-1)^{a+b} \sum_{\substack{m \geq 1 \\ m \equiv -u_1 (N)}} \sum_{r \in \Z} c_r(\phi_{a,b,\frac{m\tau - u_2}{N}}) \right)
\end{align*}

On distingue ensuite suivant la valeur de $r$, ce qui donne

\begin{align*}
E^{a,b}_{(u_1,u_2)}(\tau) & = \delta_{u_1=0} N^{-a-b-2} \left(\zeta\left(\frac{u_2}{N},a+b+2\right) + (-1)^{a+b} \zeta\left(-\frac{u_2}{N},a+b+2\right) \right) \\
& \quad + N^{-a-b-2} \sum_{\substack{m \geq 1 \\ m \equiv u_1 (N)}} (-1)^b 2i\pi \begin{pmatrix} a+b \\ a \end{pmatrix} \left(\frac{m\tau-m\overline{\tau}}{N}\right)^{-a-b-1} \\
& \quad + (-1)^{a+b} N^{-a-b-2} \sum_{\substack{m \geq 1 \\ m \equiv -u_1 (N)}} (-1)^b 2i\pi \begin{pmatrix} a+b \\ a \end{pmatrix} \left(\frac{m\tau-m\overline{\tau}}{N}\right)^{-a-b-1} \\
& \quad + N^{-a-b-2} \left( \sum_{\substack{m \geq 1 \\ m \equiv u_1 (N)}} \sum_{r \geq 1} c_r(\phi_{a,b,\frac{m\tau + u_2}{N}}) + (-1)^{a+b} \sum_{\substack{m \geq 1 \\ m \equiv -u_1 (N)}} \sum_{r \geq 1} c_r(\phi_{a,b,\frac{m\tau - u_2}{N}}) \right)\\
& \quad + N^{-a-b-2} \left( \sum_{\substack{m \geq 1 \\ m \equiv u_1 (N)}} \sum_{r \leq -1} c_r(\phi_{a,b,\frac{m\tau + u_2}{N}}) + (-1)^{a+b} \sum_{\substack{m \geq 1 \\ m \equiv -u_1 (N)}} \sum_{r \leq -1} c_r(\phi_{a,b,\frac{m\tau - u_2}{N}}) \right).
\end{align*}

En utilisant l'expression des coefficients de Fourier de $\phi_{a,b,\tau}$, il vient

\begin{align*}
E^{a,b}_{(u_1,u_2)}(\tau) & = \delta_{u_1=0} N^{-a-b-2} \left(\zeta\left(\frac{u_2}{N},a+b+2\right) + (-1)^{a+b} \zeta\left(-\frac{u_2}{N},a+b+2\right) \right) \\
& \quad + (-1)^b \frac{2i\pi}{N^{a+b+2}} \begin{pmatrix} a+b \\ a \end{pmatrix} \left( \zeta\left(\frac{u_1}{N},a+b+1\right) + (-1)^{a+b} \zeta\left(-\frac{u_1}{N},a+b+1\right)\right) (\tau-\overline{\tau})^{-a-b-1} \\
& \quad + N^{-a-b-2} \sum_{\substack{m \geq 1 \\ m \equiv u_1 (N)}} \sum_{r \geq 1} \frac{(-1)^b}{b!} 2i\pi  \sum_{j=0}^a (-1)^j \frac{(a+b-j)!}{j! (a-j)!} \left(\frac{m(\tau-\overline{\tau})}{N}\right)^{-a-b-1+j} (2i\pi r)^j e^{2i\pi r \frac{m\tau+u_2}{N}} \\
& \quad + \frac{(-1)^{a+b}}{N^{a+b+2}} \sum_{\substack{m \geq 1 \\ m \equiv -u_1 (N)}} \sum_{r \geq 1} \frac{(-1)^b}{b!} 2i\pi  \sum_{j=0}^a (-1)^j \frac{(a+b-j)!}{j! (a-j)!} \left(\frac{m(\tau-\overline{\tau})}{N}\right)^{-a-b-1+j} (2i\pi r)^j e^{2i\pi r \frac{m\tau-u_2}{N}} \\
& \quad + N^{-a-b-2} \sum_{\substack{m \geq 1 \\ m \equiv u_1 (N)}} \sum_{r \leq -1} \frac{(-1)^b}{a!} 2i\pi  \sum_{j=0}^b \frac{(a+b-j)!}{j! (b-j)!} \left(\frac{m(\tau-\overline{\tau})}{N}\right)^{-a-b-1+j} (2i\pi r)^j e^{2i\pi r(\frac{m\overline{\tau}+u_2}{N})}\\
& \quad  + \frac{(-1)^{a+b}}{N^{a+b+2}} \sum_{\substack{m \geq 1 \\ m \equiv -u_1 (N)}} \sum_{r \leq -1} \frac{(-1)^b}{a!} 2i\pi  \sum_{j=0}^b \frac{(a+b-j)!}{j! (b-j)!} \left(\frac{m(\tau-\overline{\tau})}{N}\right)^{-a-b-1+j} (2i\pi r)^j e^{2i\pi r(\frac{m\overline{\tau}-u_2}{N})}.
 \end{align*}
 
On obtient
\begin{align*}
E^{a,b}_{(u_1,u_2)}(\tau) & = \delta_{u_1=0} N^{-a-b-2} \left(\zeta\left(\frac{u_2}{N},a+b+2\right) + (-1)^{a+b} \zeta\left(-\frac{u_2}{N},a+b+2\right) \right) \\
& \quad + (-1)^b \frac{2i\pi}{N^{a+b+2}} \begin{pmatrix} a+b \\ a \end{pmatrix} \left( \zeta\left(\frac{u_1}{N},a+b+1\right) + (-1)^{a+b} \zeta\left(-\frac{u_1}{N},a+b+1\right)\right) (\tau-\overline{\tau})^{-a-b-1} \\
& \quad +  \frac{(-1)^{b+1}}{b!} \sum_{j=0}^a \frac{(a+b-j)!}{j! (a-j)!}  \left(-\frac{2i\pi}{N}\right)^{j+1} (\tau-\overline{\tau})^{-a-b-1+j} \sum_{\substack{m \geq 1 \\ m \equiv u_1 (N)}} \sum_{r \geq 1}   m^{-a-b-1+j} r^j \zeta_N^{ru_2} q^{\frac{mr}{N}} \\
& \quad +  \frac{(-1)^{a+1}}{b!} \sum_{j=0}^a \frac{(a+b-j)!}{j! (a-j)!}  \left(-\frac{2i\pi}{N}\right)^{j+1} (\tau-\overline{\tau})^{-a-b-1+j} \sum_{\substack{m \geq 1 \\ m \equiv -u_1 (N)}} \sum_{r \geq 1}   m^{-a-b-1+j} r^j \zeta_N^{-ru_2} q^{\frac{mr}{N}} \\
& \quad +  \frac{(-1)^{b+1}}{a!} \sum_{j=0}^b \frac{(a+b-j)!}{j! (b-j)!} \left(-\frac{2i\pi}{N}\right)^{j+1} (\tau-\overline{\tau})^{-a-b-1+j}\sum_{\substack{m \geq 1 \\ m \equiv u_1 (N)}} \sum_{r \geq 1} m^{-a-b-1+j}  r^j \zeta_N^{-ru_2} \overline{q}^{\frac{mr}{N}}\\
& \quad +  \frac{(-1)^{a+1}}{a!} \sum_{j=0}^b \frac{(a+b-j)!}{j! (b-j)!} \left(-\frac{2i\pi}{N}\right)^{j+1} (\tau-\overline{\tau})^{-a-b-1+j}\sum_{\substack{m \geq 1 \\ m \equiv -u_1 (N)}} \sum_{r \geq 1} m^{-a-b-1+j}  r^j \zeta_N^{ru_2} \overline{q}^{\frac{mr}{N}}
 \end{align*}
d'où le résultat.
\end{proof}

\begin{pro}\label{fourier Fab}
Pour $a,b \geq 0$ et $u_1,u_2 \in \Z/N\Z$, on a
\begin{equation}
\begin{split}
F^{a,b}_{(u_1,u_2)}(\tau) & = \hat{\zeta}\left(\frac{u_2}{N},a+b+2\right) + (-1)^{a+b} \hat{\zeta}\left(-\frac{u_2}{N},a+b+2\right)\\
& \quad + (-1)^b 2i\pi \begin{pmatrix} a+b \\ a \end{pmatrix} \delta_0(u_2) (\tau-\overline{\tau})^{-a-b-1} \left(\hat{\zeta}\left(\frac{u_1}{N},a+b+1\right)+(-1)^{a+b} \hat{\zeta}\left(-\frac{u_1}{N},a+b+1\right)\right) \\
& \quad + \frac{(-1)^{b+1}}{b!} \sum_{j=0}^{a} \frac{(a+b-j)!}{j! (a-j)!} \left(-\frac{2i\pi}{N}\right)^{j+1} (\tau-\overline{\tau})^{-a-b-1+j} N \bigl(S^{j-a-b-1,j}_{\hat{\delta}_{-u_1},\delta_{-u_2}}(\tau) + (-1)^{a+b} S^{j-a-b-1,j}_{\hat{\delta}_{u_1},\delta_{u_2}}(\tau)\bigr)\\
& \quad + \frac{(-1)^{b+1}}{a!} \sum_{j=0}^{b} \frac{(a+b-j)!}{j! (b-j)!} \left(-\frac{2i\pi}{N}\right)^{j+1} (\tau-\overline{\tau})^{-a-b-1+j} N \bigl(\overline{S}^{j-a-b-1,j}_{\hat{\delta}_{u_1},\delta_{u_2}}(\tau) + (-1)^{a+b} \overline{S}^{j-a-b-1,j}_{\hat{\delta}_{-u_1},\delta_{-u_2}}(\tau)\bigr).
\end{split}
\end{equation}
\end{pro}

\begin{proof}
On utilise la proposition \ref{fourier Eab} et l'identité (\ref{Fab Eab}).
\end{proof}

On déduit le développement de Fourier de $\Eis^n_{\mathcal{D}}(u)$ de la formule (\ref{formule EisnD Fgu}) et de la proposition \ref{fourier Fab}.

\section{Calcul du régulateur}\label{sec calcul reg}

Le calcul de l'intégrale de $\Eis^{k_1,k_2}_{\mathcal{D}}(u_1,u_2)$ le long du cycle de Shokurov $X^k \{0,\infty\}$ procède en trois étapes. Dans un premier temps, nous intégrons $\Eis^{k_1,k_2}_{\mathcal{D}}(u_1,u_2)$ sur les fibres $\gamma_{y,k}$ du cycle de Shokurov. Dans un deuxième temps, nous intégrons le résultat obtenu le long du symbole modulaire $\{0,\infty\}$, à l'aide de la méthode de Rogers-Zudilin. L'expression obtenue (\ref{expression ABCDEF}) comporte six termes : un terme principal (noté $A$) et cinq termes supplémentaires (notés $B$ à $F$) provenant des termes constants des séries d'Eisenstein. Dans un troisième temps, nous montrons que tous les termes supplémentaires se simplifient.

Soit $p : E^k(\CC) \to Y(N)(\CC)$ la projection canonique. On a une application naturelle $\nu : \h \to Y(N)(\CC)$ donnée par $\nu(\tau) = [(\tau,\sigma)]$, qui nous permet de voir le chemin $\{0,\infty\}$ comme une sous-variété de $Y(N)(\CC)$. Considérons la sous-variété $W = p^{-1}(\{0,\infty\})$ de $E^k(\CC)$. On a
\begin{equation*}
W = \bigcup_{y>0} p^{-1}(\nu(iy)) =  \bigcup_{y>0} \{iy\} \times \left(\frac{\CC}{\Z+iy\Z}\right)^k \times \{\sigma\}.
\end{equation*}

Soient $u_1=(a_1,b_1)$ et $u_2= (a_2,b_2) \in (\Z/N\Z)^2$. Nous allons calculer la restriction de $\Eis^{k_1,k_2}_{\mathcal{D}}(u_1,u_2)$ à $W$. D'après la proposition \ref{pro EisD}, on a
\begin{equation*}
\Eis^{k_1}_{\mathcal{D}}(u_1) = -\frac{k_1!(k_1+2)}{2\pi N} \Im(\tau) \sum_{a=0}^{k_1} F^{a,k_1-a}_{gu_1}(\tau) \psi_{a,k_1-a} \mod{\dd\tau,\dd\overline{\tau}}.
\end{equation*}
Comme $\Eis^{k_1}_{\mathcal{D}}(u_1)$ est invariante par $\sigma : (\tau,z,g) \mapsto (-\frac{1}{\tau},\frac{z}{\tau},\sigma g)$, il vient
\begin{equation*}
\Eis^{k_1}_{\mathcal{D}}(u_1) = \sigma^* \Eis^{k_1}_{\mathcal{D}}(u_1) = -\frac{k_1!(k_1+2)}{2\pi N} \Im(-\frac{1}{\tau}) \sum_{a=0}^{k_1} F^{a,k_1-a}_{\sigma gu_1}(-\frac{1}{\tau}) \tau^{-a} \overline{\tau}^{-k_1+a} \psi_{a,k_1-a} \mod{\dd\tau,\dd\overline{\tau}}.
\end{equation*}
En restreignant à $W$, il vient
\begin{align}\nonumber
p_1^* \Eis^{k_1}_{\mathcal{D}}(u_1) |_W & = -\frac{k_1!(k_1+2)}{2\pi N} y^{-1} \sum_{a=0}^{k_1} F^{a,k_1-a}_{-u_1}(\frac{i}{y}) (iy)^{-a} (-iy)^{-k_1+a} p_1^* \psi_{a,k_1-a} \mod{\dd y}\\
\label{p1 Eisk1}
 & =  -\frac{k_1!(k_1+2)}{2\pi N} y^{-1} (iy)^{-k_1} \sum_{a=0}^{k_1} F^{a,k_1-a}_{-u_1}(\frac{i}{y}) (-1)^{-k_1+a} p_1^* \psi_{a,k_1-a} \mod{\dd y}
\end{align}
En utilisant la proposition \ref{fourier Fab}, il vient
\begin{equation*}
\begin{split}
F^{a,k_1-a}_{-u_1}(\frac{i}{y}) & = \alpha_1 + (-1)^{k_1-a} \begin{pmatrix} k_1 \\ a \end{pmatrix} \beta_1 y^{k_1+1}\\
& \quad + \frac{(-1)^{k_1+1-a}}{(k_1-a)!} N \sum_{\ell=0}^a \frac{(k_1-\ell)!}{\ell!(a-\ell)!}(-\frac{2i\pi}{N})^{\ell+1} (\frac{2i}{y})^{-k_1-1+\ell} \left(S^{\ell-k_1-1,\ell}_{\hat{\delta}_{a_1},\delta_{b_1}}(\frac{i}{y})+(-1)^{k_1} S^{\ell-k_1-1,\ell}_{\hat{\delta}_{-a_1},\delta_{-b_1}}(\frac{i}{y})\right)\\
& \quad + \frac{(-1)^{k_1+1-a}}{a!} N \sum_{\ell=0}^{k_1-a} \frac{(k_1-\ell)!}{\ell!(k_1-a-\ell)!} (-\frac{2i\pi}{N})^{\ell+1} (\frac{2i}{y})^{-k_1-1+\ell} \left(\overline{S}^{\ell-k_1-1,\ell}_{\hat{\delta}_{-a_1},\delta_{-b_1}}(\frac{i}{y})+(-1)^{k_1} \overline{S}^{\ell-k_1-1,\ell}_{\hat{\delta}_{a_1},\delta_{b_1}}(\frac{i}{y})\right).
\end{split}
\end{equation*}
avec
\begin{align*}
\alpha_1 &  = \hat{\zeta}(-\frac{b_1}{N},k_1+2) + (-1)^{k_1} \hat{\zeta}(\frac{b_1}{N},k_1+2)\\
\beta_1 & = 2i\pi \delta_{b_1=0} (2i)^{-k_1-1} \left(\hat{\zeta}(-\frac{a_1}{N},k_1+1)+(-1)^{k_1} \hat{\zeta}(\frac{a_1}{N},k_1+1)\right).
\end{align*}
En reportant dans (\ref{p1 Eisk1}), il vient
\begin{equation}
p_1^* \Eis^{k_1}_{\mathcal{D}}(u_1) |_W = \eta_0+\eta_1
\end{equation}
avec
\begin{equation}\label{eq Ek1u1}
\begin{split}
\eta_0 & = - \frac{k_1!(k_1+2)}{2\pi N} i^{-k_1} \sum_{a=0}^{k_1} \left((-1)^{k_1-a}  \alpha_1 y^{-k_1-1} + \begin{pmatrix} k_1 \\ a \end{pmatrix} \beta_1 \right) p_1^* \psi_{a,k_1-a} \\
\eta_1 & = (-1)^{k_1+1} \frac{k_1!(k_1+2)}{2^{k_1+3} \pi} \sum_{\ell=0}^{k_1} \left( S^{\ell-k_1-1,\ell}_{\hat{\delta}_{a_1},\delta_{b_1}}(\frac{i}{y}) + (-1)^{k_1} S^{\ell-k_1-1,\ell}_{\hat{\delta}_{-a_1},\delta_{-b_1}}(\frac{i}{y})\right) \Omega_\ell \\
& \quad + (-1)^{k_1+1} \frac{k_1!(k_1+2)}{2^{k_1+3} \pi} \sum_{\ell=0}^{k_1} \left( \overline{S}^{\ell-k_1-1,\ell}_{\hat{\delta}_{-a_1},\delta_{-b_1}}(\frac{i}{y}) + (-1)^{k_1} \overline{S}^{\ell-k_1-1,\ell}_{\hat{\delta}_{a_1},\delta_{b_1}}(\frac{i}{y})\right) \overline{\Omega_\ell} \mod \dd y
\end{split}
\end{equation}
où l'on a posé, pour tout $0 \leq \ell \leq k_1$ :
\begin{equation*}
\Omega_\ell = \frac{(k_1-\ell)!}{\ell!} (\frac{4\pi}{N})^{\ell+1} y^{-\ell} \sum_{a=\ell}^{k_1} \frac{p_1^* \psi_{a,k_1-a}}{(k_1-a)!(a-\ell)!}.
\end{equation*}

D'autre part, en utilisant la proposition \ref{pro Enu}, on a
\begin{align*}
p_2^* \Eis^{k_2}_{\mathrm{hol}}(u_2) & = (-1)^{k_2+1} \frac{k_2+2}{N} (2i\pi)^{k_2} F^{(k_2+2)}_{\sigma gu_2}(\tau) \frac{\dd q}{q} \wedge p_2^* \psi_{k_2,0}\\
& = (-1)^{k_2+1} \frac{k_2+2}{N} (2i\pi)^{k_2+1} F^{(k_2+2)}_{\sigma gu_2}(\tau) \dd \tau \wedge p_2^* \psi_{k_2,0}
\end{align*}
d'où
\begin{equation*}
p_2^* \Eis^{k_2}_{\mathrm{hol}}(u_2) |_W  = i \frac{k_2+2}{N} (-2i\pi)^{k_2+1} F^{(k_2+2)}_{-u_2}(iy) \dd y \wedge p_2^* \psi_{k_2,0}
\end{equation*}
puis
\begin{equation*}
\pi_{k_2+1}(p_2^* \Eis^{k_2}_{\mathrm{hol}}(u_2)) |_W = i \frac{k_2+2}{2N} (-2i\pi)^{k_2+1} \dd y \wedge \left(F^{(k_2+2)}_{-u_2}(iy) p_2^* \psi_{k_2,0} - \overline{F}^{(k_2+2)}_{-u_2}(iy) p_2^* \psi_{0,k_2} \right).
\end{equation*}

\begin{lem}\label{lem int fibres 1}
Soit $0 \leq a \leq k_1$ un entier. On a
\begin{align}
\int_{\gamma_{y,k}} p_1^* \psi_{a,k_1-a} \wedge p_2^* \psi_{k_2,0} & = (-1)^{k_1-a} (iy)^k \\
\int_{\gamma_{y,k}} p_1^* \psi_{a,k_1-a} \wedge p_2^* \psi_{0,k_2} & = (-1)^{k-a} (iy)^k
\end{align}
\end{lem}

\begin{lem}\label{lem int fibres 2}
Pour tout entier $0 \leq \ell \leq k_1$, on a
\begin{align*}
\int_{\gamma_{y,k}} \Omega_\ell \wedge p_2^* \psi_{k_2,0} & = \begin{cases} 0 & \textrm{si } \ell <k_1\\
\frac{i^k}{k_1!} (\frac{4\pi}{N})^{k_1+1} y^{k_2} & \textrm{si } \ell=k_1.
\end{cases}\\
\int_{\gamma_{y,k}} \Omega_\ell \wedge p_2^* \psi_{0,k_2} & = \begin{cases} 0 & \textrm{si } \ell <k_1\\
(-1)^{k_2} \frac{i^k}{k_1!} (\frac{4\pi}{N})^{k_1+1} y^{k_2} & \textrm{si } \ell=k_1.
\end{cases}\\
\int_{\gamma_{y,k}} \overline{\Omega_\ell} \wedge p_2^* \psi_{k_2,0} & = \begin{cases} 0 & \textrm{si } \ell <k_1\\
(-1)^{k_1} \frac{i^k}{k_1!} (\frac{4\pi}{N})^{k_1+1} y^{k_2} & \textrm{si } \ell=k_1.
\end{cases} \\
\int_{\gamma_{y,k}} \overline{\Omega_\ell} \wedge p_2^* \psi_{0,k_2} & = \begin{cases} 0 & \textrm{si } \ell <k_1\\
(-1)^{k} \frac{i^k}{k_1!} (\frac{4\pi}{N})^{k_1+1} y^{k_2} & \textrm{si } \ell=k_1.
\end{cases}
\end{align*}
\end{lem}

\begin{proof}
On applique le lemme \ref{lem int fibres 1} avec $m=k=k_1+k_2$, ce qui donne
\begin{align*}
\int_{\gamma_{y,k}} \Omega_\ell \wedge p_2^* \psi_{k_2,0} & = \frac{(k_1-\ell)!}{\ell!} (\frac{4\pi}{N})^{\ell+1} y^{-\ell} \sum_{a=\ell}^{k_1} \frac{1}{(k_1-a)!(a-\ell)!} \int_{\gamma_{y,k}} p_1^* \psi_{a,k_1-a} \wedge p_2^* \psi_{k_2,0}\\
& = \frac{(k_1-\ell)!}{\ell!} (\frac{4\pi}{N})^{\ell+1} y^{-\ell} \sum_{a=\ell}^{k_1} \frac{(-1)^{k_1-a} (iy)^k}{(k_1-a)!(a-\ell)!}\\
& = \frac{(k_1-\ell)!}{\ell!} (\frac{4\pi}{N})^{\ell+1} y^{-\ell} (iy)^k \sum_{a=0}^{k_1-\ell} \frac{(-1)^{k_1-\ell-a}}{(k_1-\ell-a)! a!}\\
& = \begin{cases} 0 & \textrm{si } \ell <k_1\\
\frac{i^k}{k_1!} (\frac{4\pi}{N})^{k_1+1} y^{k_2} & \textrm{si } \ell=k_1.
\end{cases}
\end{align*}
Les autres cas se traitent de manière similaire.
\end{proof}

On intègre $\eta_0 \wedge \pi_{k_2+1}(p_2^* \Eis^{k_2}_{\mathrm{hol}}(u_2))$ sur les fibres de $X^k \{0,\infty\} \to \{0,\infty\}$ grâce au lemme \ref{lem int fibres 1}, ce qui donne
\begin{align}\label{int eta0}
& \int_{\gamma_{y,k}} \eta_0 \wedge \pi_{k_2+1}(p_2^* \Eis^{k_2}_{\mathrm{hol}}(u_2))\\
\nonumber & = - (-1)^{k_1} \frac{k_1!(k_1+2)}{2\pi N} i^{-k_1} \cdot i \frac{k_2+2}{2N} (-2i\pi)^{k_2+1} \dd y \sum_{a=0}^{k_1} \left((-1)^{k_1-a}  \alpha_1 y^{-k_1-1} + \begin{pmatrix} k_1 \\ a \end{pmatrix} \beta_1 \right) \cdot \\
\nonumber & \qquad \cdot \int_{\gamma_{y,k}} p_1^* \psi_{a,k_1-a} \wedge \left(F^{(k_2+2)}_{-u_2}(iy) p_2^* \psi_{k_2,0} - \overline{F}^{(k_2+2)}_{-u_2}(iy) p_2^* \psi_{0,k_2} \right)\\
\nonumber & = - (-1)^{k_1} \frac{k_1!(k_1+2)}{2\pi N} i^{-k_1} \cdot i \frac{k_2+2}{2N} (-2i\pi)^{k_2+1} \dd y \sum_{a=0}^{k_1} \left((-1)^{k_1-a}  \alpha_1 y^{-k_1-1} + \begin{pmatrix} k_1 \\ a \end{pmatrix} \beta_1 \right) \cdot \\
\nonumber & \qquad \cdot \left(F^{(k_2+2)}_{-u_2}(iy) (-1)^{k_1-a} (iy)^k - \overline{F}^{(k_2+2)}_{-u_2}(iy) (-1)^{k-a} (iy)^k \right) \\
\nonumber & = - (-1)^{k_1} \frac{k_1!(k_1+2)}{2\pi N} i^{-k_1} \cdot i \frac{k_2+2}{2N} (-2i\pi)^{k_2+1}  \sum_{a=0}^{k_1} \left((-1)^{k_1-a}  \alpha_1 y^{-k_1-1} + \begin{pmatrix} k_1 \\ a \end{pmatrix} \beta_1 \right) \cdot \\
\nonumber & \qquad \cdot (-1)^{k_1-a} (iy)^k \left(F^{(k_2+2)}_{-u_2}(iy) +(-1)^{k_2+1} \overline{F}^{(k_2+2)}_{-u_2}(iy) \right) \dd y
\end{align}
\begin{align}
\nonumber & = - (-1)^{k_1} \frac{k_1!(k_1+2)}{2\pi N} i^{k_2+1} \frac{k_2+2}{2N} (-2i\pi)^{k_2+1}  \sum_{a=0}^{k_1} \left( \alpha_1 y^{k_2-1} + (-1)^{k_1-a} \begin{pmatrix} k_1 \\ a \end{pmatrix} \beta_1 y^k \right) \cdot \\
\nonumber & \qquad \cdot \left(F^{(k_2+2)}_{-u_2}(iy) +(-1)^{k_2+1} \overline{F}^{(k_2+2)}_{-u_2}(iy) \right) \dd y\\
\nonumber & = - (-1)^{k_1} \frac{k_1!(k_1+2)(k_2+2)}{2 N^2}  (2\pi)^{k_2} \left( (k_1+1) \alpha_1 y^{k_2-1} + \delta_{k_1=0} \beta_1 y^k \right) \cdot \\
\nonumber & \qquad \cdot \left(F^{(k_2+2)}_{-u_2}(iy) +(-1)^{k_2+1} \overline{F}^{(k_2+2)}_{-u_2}(iy) \right) \dd y.
\end{align}

On intègre $\eta_1 \wedge \pi_{k_2+1}(p_2^* \Eis^{k_2}_{\mathrm{hol}}(u_2))$ sur les fibres de $X^k \{0,\infty\} \to \{0,\infty\}$ grâce au lemme \ref{lem int fibres 2}, ce qui donne
\begin{align}
\label{int eta1} & \int_{\gamma_{y,k}} \eta_1 \wedge \pi_{k_2+1}(p_2^* \Eis^{k_2}_{\mathrm{hol}}(u_2))\\
\nonumber & = - i^{k_1} \frac{(k_1+2)(k_2+2)}{4N^{k_1+2}} (2\pi)^{k_1+k_2+1} \cdot \\
\nonumber & \qquad \cdot \left(S^{-1,k_1}_{\hat{\delta}_{a_1},\delta_{b_1}}(\frac{i}{y}) + (-1)^{k_1} S^{-1,k_1}_{\hat{\delta}_{-a_1},\delta_{-b_1}}(\frac{i}{y})+ (-1)^{k_1} \overline{S}^{-1,k_1}_{\hat{\delta}_{-a_1},\delta_{-b_1}}(\frac{i}{y}) + \overline{S}^{-1,k_1}_{\hat{\delta}_{a_1},\delta_{b_1}}(\frac{i}{y})\right) \cdot \\
\nonumber & \qquad \cdot \left(F^{(k_2+2)}_{-u_2}(iy) + (-1)^{k_2+1} \overline{F}^{(k_2+2)}_{-u_2}(iy)\right)  y^{k_2} \dd y.
\end{align}
Puisque $\overline{S}^{k,\ell}_{\alpha,\beta}(\frac{i}{y}) = S^{k,\ell}_{\overline{\alpha},\overline{\beta}}(\frac{i}{y})$ et comme $\overline{\delta_a}=\delta_a$ et $\overline{\hat{\delta}_a}=\hat{\delta}_{-a}$, il vient
\begin{align*}
& S^{-1,k_1}_{\hat{\delta}_{a_1},\delta_{b_1}}(\frac{i}{y}) + (-1)^{k_1} S^{-1,k_1}_{\hat{\delta}_{-a_1},\delta_{-b_1}}(\frac{i}{y})+ (-1)^{k_1} \overline{S}^{-1,k_1}_{\hat{\delta}_{-a_1},\delta_{-b_1}}(\frac{i}{y}) + \overline{S}^{-1,k_1}_{\hat{\delta}_{a_1},\delta_{b_1}}(\frac{i}{y})\\
& = S^{-1,k_1}_{\hat{\delta}_{a_1},\delta_{b_1}}(\frac{i}{y}) + (-1)^{k_1} S^{-1,k_1}_{\hat{\delta}_{-a_1},\delta_{-b_1}}(\frac{i}{y})+ (-1)^{k_1} S^{-1,k_1}_{\hat{\delta}_{a_1},\delta_{-b_1}}(\frac{i}{y}) + S^{-1,k_1}_{\hat{\delta}_{-a_1},\delta_{b_1}}(\frac{i}{y})\\
& = S^{-1,k_1}_{\hat{\delta}_{a_1}+\hat{\delta}_{-a_1}, \delta_{b_1} + (-1)^{k_1} \delta_{-b_1}}(\frac{i}{y}).
\end{align*}
En mettant ensemble (\ref{int eta0}) et (\ref{int eta1}), il vient
\begin{align*}
& \int_{\gamma_{y,k}} p_1^* \Eis^{k_1}_{\mathcal{D}}(u_1) \wedge  \pi_{k_2+1}(p_2^* \Eis^{k_2}_{\mathrm{hol}}(u_2))\\ 
& = (-1)^{k_1+1} \frac{k_1!(k_1+2)(k_2+2)}{2 N^2}  (2\pi)^{k_2} \left( (k_1+1) \alpha_1 y^{k_2-1} + \delta_{k_1=0} \beta_1 y^{k_2} \right) \left(F^{(k_2+2)}_{-u_2}(iy) +(-1)^{k_2+1} \overline{F}^{(k_2+2)}_{-u_2}(iy) \right) \dd y\\
& \quad - i^{k_1} \frac{(k_1+2)(k_2+2)}{4N^{k_1+2}} (2\pi)^{k_1+k_2+1} S^{-1,k_1}_{\hat{\delta}_{a_1}+\hat{\delta}_{-a_1}, \delta_{b_1} + (-1)^{k_1} \delta_{-b_1}}(\frac{i}{y}) \left(F^{(k_2+2)}_{-u_2}(iy) + (-1)^{k_2+1} \overline{F}^{(k_2+2)}_{-u_2}(iy)\right) y^{k_2} \dd y
\end{align*}

D'autre part, d'après le lemme \ref{lem Fk}, le terme constant de la série d'Eisenstein $F^{(k_2+2)}_{-u_2}$ est réel, et on a
\begin{align*}
& F^{(k_2+2)}_{-u_2}(iy) + (-1)^{k_2+1} \overline{F}^{(k_2+2)}_{-u_2}(iy) \\
& = \zeta(-\frac{a_2}{N},-k_2-1) + (-1)^{k_2+1} \zeta(-\frac{a_2}{N},-k_2-1)\\
& \qquad + N^{-k_2-1} \left(\sum_{\substack{m,n \geq 1\\ n \equiv -a_2(N)}} \zeta_N^{-b_2 m} n^{k_2+1} e^{-\frac{2\pi mny}{N}} + (-1)^{k_2} \sum_{\substack{m,n \geq 1\\ n \equiv a_2(N)}} \zeta_N^{b_2 m} n^{k_2+1} e^{-\frac{2\pi mny}{N}}\right)\\
& \qquad + (-1)^{k_2+1} N^{-k_2-1} \left(\sum_{\substack{m,n \geq 1\\ n \equiv -a_2(N)}} \zeta_N^{b_2 m} n^{k_2+1} e^{-\frac{2\pi mny}{N}} + (-1)^{k_2} \sum_{\substack{m,n \geq 1\\ n \equiv a_2(N)}} \zeta_N^{-b_2 m} n^{k_2+1} e^{-\frac{2\pi mny}{N}}\right)\\
& = (1+(-1)^{k_2+1}) \zeta(-\frac{a_2}{N},-k_2-1) + N^{-k_2-1} S^{0,k_2+1}_{\hat{\delta}_{b_2}+(-1)^{k_2+1} \hat{\delta}_{-b_2}, \delta_{-a_2}-\delta_{a_2}}(iy).
\end{align*}

\begin{lem}\label{lem mellin F}
Pour $s \in \CC$, $\Re(s)<0$, la fonction $y \mapsto \left(F^{(k_2+2)}_{-u_2}(iy)+(-1)^{k_2+1}\overline{F}^{(k_2+2)}_{-u_2}(iy)\right) y^{s-1}$ est intégrable sur $]0,+\infty[$, et on a
\begin{equation*}
\begin{split}
& \int_0^\infty \left(F^{(k_2+2)}_{-u_2}(iy)+(-1)^{k_2+1}\overline{F}^{(k_2+2)}_{-u_2}(iy)\right) y^s \frac{\dd y}{y} = i^{k_2+2} (2\pi)^{s-k_2-2} \Gamma(-s+k_2+2) \cdot \\
& \qquad \cdot \left(\hat{\zeta}(\frac{a_2}{N},-s+k_2+2)-\hat{\zeta}(-\frac{a_2}{N},-s+k_2+2)\right) \left(\zeta(-\frac{b_2}{N},-s+1)+(-1)^{k_2+1} \zeta(\frac{b_2}{N},-s+1)\right).
\end{split}
\end{equation*}
\end{lem}

\begin{proof}
L'hypothèse $\Re(s)<0$ entraîne que l'intégrale converge lorsque $y \to +\infty$. On a $\overline{F}^{(k_2+2)}_{-a_2,-b_2}(iy) = F^{(k_2+2)}_{-a_2,b_2}(iy)$. D'après le lemme \ref{lem Fk SL2} appliqué avec $g=\sigma$, on a
\begin{equation*}
F^{(k_2+2)}_{-a_2,-b_2}(\frac{i}{y}) = (iy)^{k_2+2} F^{(k_2+2)}_{-b_2,a_2}(iy).
\end{equation*}
On fait le changement de variables $y \mapsto 1/y$ dans l'intégrale, ce qui donne
\begin{align*}
& \int_0^\infty \left(F^{(k_2+2)}_{-u_2}(iy)+(-1)^{k_2+1}\overline{F}^{(k_2+2)}_{-u_2}(iy)\right) y^s \frac{\dd y}{y}\\
& =  \int_0^\infty \left(F^{(k_2+2)}_{-a_2,-b_2}(\frac{i}{y})+(-1)^{k_2+1} F^{(k_2+2)}_{-a_2,b_2}(\frac{i}{y})\right) y^{-s} \frac{\dd y}{y}\\
& = \int_0^\infty \left((iy)^{k_2+2} F^{(k_2+2)}_{-b_2,a_2}(iy)+(-1)^{k_2+1} (iy)^{k_2+2} F^{(k_2+2)}_{b_2,a_2}(iy)\right) y^{-s} \frac{\dd y}{y}\\
& = i^{k_2+2} \int_0^\infty \left( F^{(k_2+2)}_{-b_2,a_2}(iy)+(-1)^{k_2+1} F^{(k_2+2)}_{b_2,a_2}(iy)\right) y^{-s+k_2+2} \frac{\dd y}{y}\\
\end{align*}
Le terme constant de la série d'Eisenstein $F^{(k_2+2)}_{-b_2,a_2} +(-1)^{k_2+1} F^{(k_2+2)}_{b_2,a_2}$ est nul, de sorte que l'intégrale de départ est convergente pour $\Re(s)<0$. D'après le lemme \ref{lem Fk}, on a
\begin{align*}
& F^{(k_2+2)}_{-b_2,a_2}(iy)+(-1)^{k_2+1} F^{(k_2+2)}_{b_2,a_2}(iy)\\
& = N^{-k_2-1} \left(S^{0,k_2+1}_{\hat{\delta}_{-a_2},\delta_{-b_2}}(iy) + (-1)^{k_2} S^{0,k_2+1}_{\hat{\delta}_{a_2},\delta_{b_2}}(iy) + (-1)^{k_2+1} S^{0,k_2+1}_{\hat{\delta}_{-a_2},\delta_{b_2}}(iy) - S^{0,k_2+1}_{\hat{\delta}_{a_2},\delta_{-b_2}}(iy) \right)\\
& = N^{-k_2-1} S^{0,k_2+1}_{\hat{\delta}_{-a_2}-\hat{\delta}_{a_2}, \delta_{-b_2}+(-1)^{k_2+1} \delta_{b_2}}(iy).
\end{align*}
Grâce au lemme \ref{lem mellin Skl}, on en déduit
\begin{align*}
& i^{k_2+2} \int_0^\infty \left( F^{(k_2+2)}_{-b_2,a_2}(iy)+(-1)^{k_2+1} F^{(k_2+2)}_{b_2,a_2}(iy)\right) y^{-s+k_2+2} \frac{\dd y}{y}\\
& = \frac{i^{k_2+2}}{N^{k_2+1}} \int_0^\infty S^{0,k_2+1}_{\hat{\delta}_{-a_2}-\hat{\delta}_{a_2}, \delta_{-b_2}+(-1)^{k_2+1} \delta_{b_2}}(iy) y^{-s+k_2+2} \frac{\dd y}{y}\\
& = \frac{i^{k_2+2}}{N^{k_2+1}} (\frac{2\pi}{N})^{s-k_2-2} \Gamma(s-k_2+2) L(\hat{\delta}_{-a_2}-\hat{\delta}_{a_2},-s+k_2+2) L(\delta_{-b_2}+(-1)^{k_2+1} \delta_{b_2},-s+1)\\
& = i^{k_2+2} (2\pi)^{s-k_2-2} \Gamma(-s+k_2+2) \left(\hat{\zeta}(\frac{a_2}{N},-s+k_2+2)-\hat{\zeta}(-\frac{a_2}{N},-s+k_2+2)\right) \cdot \\
& \qquad \qquad \cdot \left(\zeta(-\frac{b_2}{N},-s+1)+(-1)^{k_2+1} \zeta(\frac{b_2}{N},-s+1)\right).
\end{align*}
\end{proof}

Pour $\Re(s) \ll 0$, on en déduit
\begin{align*}
& \int_{X^k \{0,\infty\}} y^s p_1^* \Eis^{k_1}_{\mathcal{D}}(u_1) \wedge \pi_{k_2+1}(p_2^* \Eis^{k_2}_{\mathrm{hol}}(u_2))\\
& = (-1)^{k_1+1} k_1! \frac{(k_1+2)(k_2+2)}{2N^2} (2\pi)^{k_2} \int_0^\infty \left(F^{(k_2+2)}_{-u_2}(iy)+(-1)^{k_2+1} \overline{F}^{(k_2+2)}_{-u_2}(iy)\right) \cdot \\
& \qquad \qquad \cdot \left((k_1+1)\alpha_1 y^{k_2-1} + \delta_{k_1=0} \beta_1 y^{k_2} \right) y^s \dd y\\
& \quad - i^{k_1} \frac{(k_1+2)(k_2+2)}{4N^{k_1+2}} (2\pi)^{k_1+k_2+1} \int_0^\infty S^{-1,k_1}_{\hat{\delta}_{a_1}+\hat{\delta}_{-a_1}, \delta_{b_1} + (-1)^{k_1} \delta_{-b_1}}(\frac{i}{y}) \cdot \\
& \qquad \qquad \cdot \left((1+(-1)^{k_2+1}) \zeta(-\frac{a_2}{N},-k_2-1) + N^{-k_2-1} S^{0,k_2+1}_{\hat{\delta}_{b_2}+(-1)^{k_2+1} \hat{\delta}_{-b_2}, \delta_{-a_2}-\delta_{a_2}}(iy)\right) y^{s+k_2} \dd y.
\end{align*}

En appliquant les lemmes \ref{lem mellin Skl}, \ref{lem mellin F} et la proposition \ref{pro RZ trick}, on obtient
\begin{align}
\nonumber & \int_{X^k \{0,\infty\}} y^s p_1^* \Eis^{k_1}_{\mathcal{D}}(u_1) \wedge \pi_{k_2+1}(p_2^* \Eis^{k_2}_{\mathrm{hol}}(u_2))\\
\nonumber & = (-1)^{k_1} \frac{(k_1+2)!(k_2+2)}{2N^2} (2i\pi)^{k_2} \alpha_1 (2\pi)^{s-2} \Gamma(-s+2) \cdot\\
\nonumber & \qquad \cdot \left(\hat{\zeta}(\frac{a_2}{N},-s+2)-\hat{\zeta}(-\frac{a_2}{N},-s+2)\right) \left(\zeta(-\frac{b_2}{N},-s-k_2+1)+(-1)^{k_2+1} \zeta(\frac{b_2}{N},-s-k_2+1)\right)\\
\nonumber & \quad + \delta_{k_1=0} \beta_1 \frac{k_2+2}{N^2} (2i\pi)^{k_2} (2\pi)^{s-1} \Gamma(-s+1) \cdot\\
\nonumber & \qquad \cdot \left(\hat{\zeta}(\frac{a_2}{N},-s+1)-\hat{\zeta}(-\frac{a_2}{N},-s+1)\right) \left(\zeta(-\frac{b_2}{N},-s-k_2)+(-1)^{k_2+1} \zeta(\frac{b_2}{N},-s-k_2)\right)\\
\nonumber & \quad - i^{k_1} \frac{(k_1+2)(k_2+2)}{4N^{k_1+2}} (2\pi)^{k_1+k_2+1} (1+(-1)^{k_2+1}) \zeta(-\frac{a_2}{N},-k_2-1) (\frac{2\pi}{N})^{s+k_2+1} \Gamma(-s-k_2-1) \cdot \\
\nonumber & \qquad \cdot \left(\hat{\zeta}(\frac{a_1}{N},-s-k_2)+\hat{\zeta}(-\frac{a_1}{N},-s-k_2)\right) N^{s+k+1} \left(\zeta(\frac{b_1}{N},-s-k-1) + (-1)^{k_1} \zeta(-\frac{b_1}{N},-s-k-1)\right)\\
\nonumber & \quad - i^{k_1} \frac{(k_1+2)(k_2+2)}{4N^{k_1+2}} (2\pi)^{k_1+k_2+1} N^{-k_2-1} \cdot \\
\label{integrale Sk1k2} & \qquad \cdot \int_0^\infty S^{s+k_2,0}_{\hat{\delta}_{a_1}+\hat{\delta}_{-a_1}, \hat{\delta}_{b_2}+(-1)^{k_2+1} \hat{\delta}_{-b_2}}(iy) S^{k_1,-s}_{\delta_{b_1} + (-1)^{k_1} \delta_{-b_1}, \delta_{-a_2}-\delta_{a_2}}(\frac{i}{y}) y^{s+k_2} \dd y.
\end{align}

L'expression précédente est une fonction méromorphe de $s$, et tous les termes qui la composent sont holomorphes au voisinage de $s=0$, excepté les termes suivants :
\begin{enumerate}
\item L'expression $\zeta(\pm \frac{b_2}{N},-s-k_2+1)$ a un pôle simple lorsque $k_2=0$. Mais dans ce cas
\begin{equation*}
\zeta(-\frac{b_2}{N},-s-k_2+1)+(-1)^{k_2+1} \zeta(\frac{b_2}{N},-s-k_2+1) = \zeta(-\frac{b_2}{N},-s+1) - \zeta(\frac{b_2}{N},-s+1)
\end{equation*}
est holomorphe en $s=0$ ; sa valeur est $\zeta^*(-\frac{b_2}{N},1) - \zeta^*(\frac{b_2}{N},1)$.
\item Le facteur $\Gamma(-s-k_2-1)$ a un pôle simple. Mais ce terme n'intervient que si $k_2$ est impair, et dans ce cas on a
\begin{equation*}
\zeta(\frac{b_1}{N},-k-1) + (-1)^{k_1} \zeta(-\frac{b_1}{N},-k-1) = -\frac{B_{k+2}(\{\frac{b_1}{N}\})}{k+2} + (-1)^{k} \frac{B_{k+2}(\{-\frac{b_1}{N}\})}{k+2} = 0.
\end{equation*}
\end{enumerate}

L'expression précédente est donc holomorphe en $s=0$. Nous allons calculer l'intégrale (\ref{integrale Sk1k2}) pour $s=0$, soit
\begin{equation*}
I := \int_0^\infty S^{k_2,0}_{\hat{\delta}_{a_1}+\hat{\delta}_{-a_1}, \hat{\delta}_{b_2}+(-1)^{k_2+1} \hat{\delta}_{-b_2}}(iy) S^{k_1,0}_{\delta_{b_1} + (-1)^{k_1} \delta_{-b_1}, \delta_{-a_2}-\delta_{a_2}}(\frac{i}{y}) y^{k_2} \dd y.
\end{equation*}
Pour cela, nous aurons besoin du lemme suivant \cite[Lemmas 8 and 9]{brunault:reg_siegel}.

\begin{lem}\label{lem fgh}
Soit $f = \sum_{n=0}^\infty a_n q^n \in M_k(\Gamma_1(N))$ et $g = \sum_{n=0}^\infty b_n q^n \in M_{\ell}(\Gamma_1(N))$ avec $k,\ell \geq 1$. Soit $h=W_N(g)$. Alors on a
\begin{equation}\label{eq fgh}
N^{s/2} \int_0^\infty f^*(iy) g^*\bigl(\frac{i}{Ny}\bigr) y^s \frac{\dd y}{y} = \Lambda(fh,s+\ell) - a_0 \Lambda(h,s+\ell) - b_0 \Lambda(f,s).
\end{equation}
En particulier, on a
\begin{equation}\label{eq fgh 2}
N^{k/2} \int_0^\infty f^*(iy) g^*\bigl(\frac{i}{Ny}\bigr) y^k \frac{\dd y}{y} = \Lambda^*(fh,k+\ell) - a_0 \Lambda(h,k+\ell) - b_0 \Lambda^*(f,k).
\end{equation}
\end{lem}

En faisant le changement de variables $y \to Ny$ dans $I$, on a
\begin{equation*}
I = N^{k_2+1} \int_0^\infty S^{k_2,0}_{\hat{\delta}_{a_1}+\hat{\delta}_{-a_1}, \hat{\delta}_{b_2}+(-1)^{k_2+1} \hat{\delta}_{-b_2}}(iNy) S^{k_1,0}_{\delta_{b_1} + (-1)^{k_1} \delta_{-b_1}, \delta_{-a_2}-\delta_{a_2}}(\frac{i}{Ny}) y^{k_2} \dd y.
\end{equation*}
Par définition, on a
\begin{align*}
S^{k_2,0}_{\hat{\delta}_{a_1}+\hat{\delta}_{-a_1}, \hat{\delta}_{b_2}+(-1)^{k_2+1} \hat{\delta}_{-b_2}}(iNy) & = H^{(k_2+1),*}_{b_2,a_1}(iy)+H^{(k_2+1),*}_{b_2,-a_1}(iy)\\
S^{k_1,0}_{\delta_{b_1} + (-1)^{k_1} \delta_{-b_1}, \delta_{-a_2}-\delta_{a_2}}(\frac{i}{Ny}) & = G^{(k_1+1),*}_{b_1,-a_2}(\frac{i}{N^2 y}) - G^{(k_1+1),*}_{b_1,a_2}(\frac{i}{N^2 y}).
\end{align*}
En utilisant le lemme \ref{lem fgh} avec $f=H^{(k_2+1)}_{b_2,a_1}+H^{(k_2+1)}_{b_2,-a_1}$ et $g=G^{(k_1+1)}_{b_1,-a_2} - G^{(k_1+1)}_{b_1,a_2}$, il vient
\begin{equation*}
I=I_1+I_2+I_3
\end{equation*}
avec
\begin{align*}
I_1 & = \Lambda^*(f \cdot W_{N^2}(g),k+2)\\
& = \Lambda^*(W_{N^2}(f) \cdot g,0)\\
& = i^{-k_2-1} N \Lambda^*((G^{(k_2+1)}_{b_2,a_1}+G^{(k_2+1)}_{b_2,-a_1})(G^{(k_1+1)}_{b_1,-a_2} - G^{(k_1+1)}_{b_1,a_2}),0)\\
I_2 & = - a_0(f) \Lambda(W_{N^2}(g),k+2)\\
I_3 & = - a_0(g) \Lambda^*(f,k_2+1).
\end{align*}

\emph{Calcul de $I_2$.} Grâce aux lemmes \ref{lem W GH} et \ref{lem L Hkab}, on a
\begin{align*}
\Lambda(W_{N^2}(g),k+2) & = \frac{i^{k_1+1}}{N} \Lambda(H^{(k_1+1)}_{b_1,-a_2}-H^{(k_1+1)}_{b_1,a_2},k+2) \\
& = \frac{i^{k_1+1}}{N} N^{k+2} (2\pi)^{-k-2} \Gamma(k+2) \cdot \\
& \quad \left( \hat{\zeta}(-\frac{b_1}{N},k+2) \hat{\zeta}(\frac{a_2}{N},k_2+2) + (-1)^{k_1+1} \hat{\zeta}(\frac{b_1}{N},k+2) \hat{\zeta}(-\frac{a_2}{N},k_2+2) \right.\\
& \qquad \left. - \hat{\zeta}(-\frac{b_1}{N},k+2) \hat{\zeta}(-\frac{a_2}{N},k_2+2) + (-1)^{k_1} \hat{\zeta}(\frac{b_1}{N},k+2) \hat{\zeta}(\frac{a_2}{N},k_2+2) \right)\\
& = i^{k_1+1} N^{k+1} (2\pi)^{-k-2} (k+1)! \left(\hat{\zeta}(-\frac{b_1}{N},k+2)+(-1)^{k_1} \hat{\zeta}(\frac{b_1}{N},k+2) \right) \\
& \qquad \left(\hat{\zeta}(\frac{a_2}{N},k_2+2)-\hat{\zeta}(-\frac{a_2}{N},k_2+2)\right).
\end{align*}
On obtient donc
\begin{align*}
I_2 & = - i^{k_1+1} N^{k+1} (2\pi)^{-k-2} (k+1)! a_0(H^{(k_2+1)}_{b_2,a_1}+H^{(k_2+1)}_{b_2,-a_1}) \cdot \\
& \qquad \left(\hat{\zeta}(-\frac{b_1}{N},k+2)+(-1)^{k_1} \hat{\zeta}(\frac{b_1}{N},k+2) \right) \left(\hat{\zeta}(\frac{a_2}{N},k_2+2)-\hat{\zeta}(-\frac{a_2}{N},k_2+2)\right).
\end{align*}

\emph{Calcul de $I_3$.} Si $k_1 \geq 1$ et $a_2 \neq 0$, alors $a_0(g)=0$. Si $a_2=0$, alors $g=0$. On peut donc supposer $k_1=0$ et $a_2 \neq 0$. Si $b_1 \neq 0$, alors $a_0(g)=0$. On peut donc supposer $b_1=0$. Alors
\begin{equation*}
a_0(g) = \left(\frac12-\left\{\frac{-a_2}{N}\right\}\right)-\left(\frac12 - \left\{\frac{a_2}{N}\right\}\right) = \left\{ \frac{a_2}{N} \right\} - \left\{-\frac{a_2}{N}\right\} = 2 \left\{\frac{a_2}{N}\right\}-1.
\end{equation*}
De plus
\begin{align*}
\Lambda(f,s) & = N^s (2\pi)^{-s} \Gamma(s) \left( \hat{\zeta}(-\frac{b_2}{N},s)\hat{\zeta}(-\frac{a_1}{N},s-k_2) + (-1)^{k_2+1} \hat{\zeta}(\frac{b_2}{N},s)\hat{\zeta}(\frac{a_1}{N},s-k_2) \right. \\
& \qquad \qquad \qquad \qquad \left. + \hat{\zeta}(-\frac{b_2}{N},s)\hat{\zeta}(\frac{a_1}{N},s-k_2)+ (-1)^{k_2+1} \hat{\zeta}(\frac{b_2}{N},s)\hat{\zeta}(-\frac{a_1}{N},s-k_2)\right)\\
& = N^s (2\pi)^{-s} \Gamma(s) \left( \hat{\zeta}(-\frac{b_2}{N},s)+(-1)^{k_2+1} \hat{\zeta}(\frac{b_2}{N},s) \right) \left( \hat{\zeta}(-\frac{a_1}{N},s-k_2) + \hat{\zeta}(\frac{a_1}{N},s-k_2)\right).
\end{align*}
Comme $a_1 \neq 0$, les fonctions $\hat{\zeta}(\frac{a_1}{N},s)$ et $\hat{\zeta}(-\frac{a_1}{N},s)$ sont holomorphes sur $\CC$, d'où
\begin{align*}
\Lambda^*(f,k_2+1) = \Lambda(f,k_2+1) = N^{k_2+1} (2\pi)^{-k_2-1} k_2! & \left( \hat{\zeta}(-\frac{b_2}{N},k_2+1)+(-1)^{k_2+1} \hat{\zeta}(\frac{b_2}{N},k_2+1) \right) \cdot \\
& \qquad \cdot \left( \hat{\zeta}(\frac{a_1}{N},1) + \hat{\zeta}(-\frac{a_1}{N},1)\right).
\end{align*}
On a alors
\begin{align*}
I_3 & = \delta_{k_1=0} \delta_{a_2 \neq 0} \delta_{b_1=0} N^{k_2+1} (2\pi)^{-k_2-1} k_2! \left(1-2\left\{\frac{a_2}{N}\right\}\right) \cdot \\
& \qquad \cdot \left( \hat{\zeta}(-\frac{b_2}{N},k_2+1)+(-1)^{k_2+1} \hat{\zeta}(\frac{b_2}{N},k_2+1) \right) \left( \hat{\zeta}(\frac{a_1}{N},1) + \hat{\zeta}(-\frac{a_1}{N},1)\right)
\end{align*}

Alors l'intégrale régularisée de $p_1^* \Eis^{k_1}_{\mathcal{D}}(u_1) \wedge \pi_{k_2+1}(p_2^* \Eis^{k_2}_{\mathrm{hol}}(u_2))$ le long de $X^k \{0,\infty\}$ est donnée par
\begin{equation}\label{expression ABCDEF}
\int_{X^k \{0,\infty\}}^{*} p_1^* \Eis^{k_1}_{\mathcal{D}}(u_1) \wedge \pi_{k_2+1}(p_2^* \Eis^{k_2}_{\mathrm{hol}}(u_2))  = A+B+C+D+E+F
\end{equation}
avec
\begin{flalign*}
A & = A^{k_1,k_2}(u_1,u_2) \\
& = i^{k_1-k_2+1} \frac{(k_1+2)(k_2+2)}{4N^{k_1+k_2+2}} (2\pi)^{k_1+k_2+1} \Lambda^*((G^{(k_2+1)}_{b_2,a_1}+G^{(k_2+1)}_{b_2,-a_1})(G^{(k_1+1)}_{b_1,-a_2} - G^{(k_1+1)}_{b_1,a_2}),0)\\
B & = B^{k_1,k_2}(u_1,u_2) \\
& = (-1)^{k_1} \frac{(k_1+2)! (k_2+2)}{8\pi^2 N^2} (2i\pi)^{k_2} \left(\hat{\zeta}(-\frac{b_1}{N},k_1+2) + (-1)^{k_1} \hat{\zeta}(\frac{b_1}{N},k_1+2) \right)\cdot\\
& \qquad \cdot \left(\hat{\zeta}(\frac{a_2}{N},2)-\hat{\zeta}(-\frac{a_2}{N},2)\right) \left(\zeta^{(*)}(-\frac{b_2}{N},-k_2+1)+(-1)^{k_2+1} \zeta^{(*)}(\frac{b_2}{N},-k_2+1)\right)\\
C & = C^{k_1,k_2}(u_1,u_2) \\
& = \delta_{k_1=0} \delta_{b_1=0} \delta_{a_2 \neq 0} \frac{k_2+2}{2N^2} (2i\pi)^{k_2} \left(\hat{\zeta}(\frac{a_1}{N},1)+\hat{\zeta}(-\frac{a_1}{N},1)\right) \cdot\\
& \qquad \cdot \left(\hat{\zeta}(\frac{a_2}{N},1)-\hat{\zeta}(-\frac{a_2}{N},1)\right) \left(\zeta(-\frac{b_2}{N},-k_2)+(-1)^{k_2+1} \zeta(\frac{b_2}{N},-k_2)\right)\\
D & = D^{k_1,k_2}(u_1,u_2) \\
& = -\delta_{k_2\equiv 1(2)} i^{k_1} \frac{(k_1+2)(k_2+2)}{2N^2} (2\pi)^{k_1+2k_2+2} \left(\hat{\zeta}(\frac{a_1}{N},-k_2)+\hat{\zeta}(-\frac{a_1}{N},-k_2)\right) \cdot \\
& \qquad \cdot \zeta(-\frac{a_2}{N},-k_2-1) \cdot \lim_{s \to 0} \left(\Gamma(s-k_2-1) \left(\zeta(\frac{b_1}{N},s-k-1) + (-1)^{k_1} \zeta(-\frac{b_1}{N},s-k-1)\right)\right)\\
E & = E^{k_1,k_2}(u_1,u_2) \\
& = - i^{k_1} \frac{(k_1+2) (k_2+2)}{4N^{k_1+2}} (2\pi)^{k_1+k_2+1} N^{-k_2-1} I_2\\
& = (-1)^{k_1} i \frac{(k_1+2)(k_2+2)(k+1)!}{8\pi N^2} a_0(H^{(k_2+1)}_{b_2,a_1}+H^{(k_2+1)}_{b_2,-a_1}) \cdot \\
& \qquad \cdot \left(\hat{\zeta}(-\frac{b_1}{N},k+2)+(-1)^{k_1} \hat{\zeta}(\frac{b_1}{N},k+2) \right) \left(\hat{\zeta}(\frac{a_2}{N},k_2+2)-\hat{\zeta}(-\frac{a_2}{N},k_2+2)\right)
\end{flalign*}
\begin{flalign*}
F & = F^{k_1,k_2}(u_1,u_2)\\
& = - i^{k_1} \frac{(k_1+2) (k_2+2)}{4N^{k_1+2}} (2\pi)^{k_1+k_2+1} N^{-k_2-1} I_3\\
& = \delta_{k_1=0} \delta_{b_1=0} \delta_{a_2 \neq 0}  \frac{k_2! (k_2+2)}{2N^2} \left(2\left\{\frac{a_2}{N}\right\}-1\right) \cdot \\
& \qquad \cdot \left( \hat{\zeta}(-\frac{b_2}{N},k_2+1)+(-1)^{k_2+1} \hat{\zeta}(\frac{b_2}{N},k_2+1) \right) \left( \hat{\zeta}(\frac{a_1}{N},1) + \hat{\zeta}(-\frac{a_1}{N},1)\right).
\end{flalign*}
Dans la formule pour $B$, le symbole $\zeta^{(*)}$ indique que l'on prend la valeur régularisée lorsque $k_2=0$.

Montrons maintenant que les termes $C$ et $F$ se simplifient. D'après la formule (\ref{hatzeta 1}), on a
\begin{equation}\label{eq C1}
\hat{\zeta}(\frac{a_2}{N},1)-\hat{\zeta}(-\frac{a_2}{N},1) = 2i\pi \left(\frac12-\left\{\frac{a_2}{N}\right\}\right).
\end{equation}
D'autre part, en appliquant la formule de Hurwitz (\ref{hurwitz 1}) en $s=k_2+1$, il vient
\begin{align*}
\zeta(-\frac{b_2}{N},-k_2) & = \frac{k_2!}{(2\pi)^{k_2+1}} \left( (-i)^{k_2+1} \hat{\zeta}(-\frac{b_2}{N},k_2+1) + i^{k_2+1} \hat{\zeta}(\frac{b_2}{N},k_2+1) \right)\\
& = (-i)^{k_2+1} \frac{k_2!}{(2\pi)^{k_2+1}} \left( \hat{\zeta}(-\frac{b_2}{N},k_2+1) + (-1)^{k_2+1} \hat{\zeta}(\frac{b_2}{N},k_2+1) \right).
\end{align*}
De même, on a
\begin{equation*}
\zeta(\frac{b_2}{N},-k_2) = (-i)^{k_2+1} \frac{k_2!}{(2\pi)^{k_2+1}} \left( \hat{\zeta}(\frac{b_2}{N},k_2+1) + (-1)^{k_2+1} \hat{\zeta}(-\frac{b_2}{N},k_2+1) \right).
\end{equation*}
On en déduit
\begin{equation}\label{eq C2}
\zeta(-\frac{b_2}{N},-k_2) + (-1)^{k_2+1} \zeta(\frac{b_2}{N},-k_2) = 2 \frac{(-i)^{k_2+1} k_2!}{(2\pi)^{k_2+1}} \left(  \hat{\zeta}(-\frac{b_2}{N},k_2+1) + (-1)^{k_2+1} \hat{\zeta}(\frac{b_2}{N},k_2+1)  \right).
\end{equation}
En reportant (\ref{eq C1}) et (\ref{eq C2}) dans le terme $C$, on obtient
\begin{align*}
C & = \delta_{k_1=0} \delta_{b_1=0} \delta_{a_2 \neq 0} \frac{k_2+2}{2N^2} (2i\pi)^{k_2} \left(\hat{\zeta}(\frac{a_1}{N},1)+\hat{\zeta}(-\frac{a_1}{N},1)\right) \cdot\\
& \qquad \cdot 2i\pi \left(\frac12-\left\{\frac{a_2}{N}\right\}\right) 2 \frac{(-i)^{k_2+1} k_2!}{(2\pi)^{k_2+1}} \left(  \hat{\zeta}(-\frac{b_2}{N},k_2+1) + (-1)^{k_2+1} \hat{\zeta}(\frac{b_2}{N},k_2+1)  \right)\\
& = \delta_{k_1=0} \delta_{b_1=0} \delta_{a_2 \neq 0} \frac{k_2! (k_2+2)}{2N^2} \left(\hat{\zeta}(\frac{a_1}{N},1)+\hat{\zeta}(-\frac{a_1}{N},1)\right) \cdot\\
& \qquad \cdot \left(1-2\left\{\frac{a_2}{N}\right\}\right) \left( \hat{\zeta}(-\frac{b_2}{N},k_2+1) + (-1)^{k_2+1} \hat{\zeta}(\frac{b_2}{N},k_2+1) \right)\\
& = - F.
\end{align*}

Nous allons maintenant simplifier d'autres termes, en distinguant les cas suivant $k_2$.

\emph{Premier cas : $k_2=0$.}

On a alors $D=0$. Nous allons montrer que $B+E=0$. On a
\begin{align*}
B & = (-1)^{k_1} \frac{(k_1+2)!}{4\pi^2 N^2} \left(\hat{\zeta}(-\frac{b_1}{N},k_1+2) + (-1)^{k_1} \hat{\zeta}(\frac{b_1}{N},k_1+2) \right)\cdot\\
& \qquad \cdot \left(\hat{\zeta}(\frac{a_2}{N},2)-\hat{\zeta}(-\frac{a_2}{N},2)\right) \left(\zeta^{*}(-\frac{b_2}{N},1)- \zeta^{*}(\frac{b_2}{N},1)\right)
\end{align*}
et
\begin{align*}
E & = (-1)^{k_1} i \frac{(k_1+2)!}{4\pi N^2} a_0(H^{(1)}_{b_2,a_1}+H^{(1)}_{b_2,-a_1}) \cdot \\
& \qquad \cdot \left(\hat{\zeta}(-\frac{b_1}{N},k_1+2)+(-1)^{k_1} \hat{\zeta}(\frac{b_1}{N},k_1+2) \right) \left(\hat{\zeta}(\frac{a_2}{N},2)-\hat{\zeta}(-\frac{a_2}{N},2)\right).
\end{align*}
Si $b_2=0$, alors $H^{(1)}_{0,a_1}+H^{(1)}_{0,-a_1}=0$ et donc $B=E=0$. Nous pouvons donc supposer $b_2 \neq 0$. D'après la formule (\ref{zeta* 1}), on a
\begin{equation*}
\zeta^{*}(-\frac{b_2}{N},1)- \zeta^{*}(\frac{b_2}{N},1) = i\pi \frac{\zeta_N^{-b_2}+1}{\zeta_N^{-b_2}-1} = i\pi \frac{1+\zeta_N^{b_2}}{1-\zeta_N^{b_2}}.
\end{equation*}
D'après la définition \ref{def Hkab} et en utilisant l'identité
\begin{equation*}
\frac{1+\zeta_N^a}{1-\zeta_N^a} + \frac{1+\zeta_N^{-a}}{1-\zeta_N^{-a}} = 0 \qquad (a \in \Z/N\Z -\{0\}),
\end{equation*}
il vient
\begin{equation*}
a_0(H^{(1)}_{b_2,a_1}+H^{(1)}_{b_2,-a_1}) = - \frac{1+\zeta_N^{b_2}}{1-\zeta_N^{b_2}}.
\end{equation*}
En reportant dans $B$ et $E$, on obtient $B+E=0$. Au final, on a donc dans ce cas
\begin{align*}
& \int_{X^k \{0,\infty\}}^{*} p_1^* \Eis^{k_1}_{\mathcal{D}}(u_1) \wedge \pi_{k_2+1}(p_2^* \Eis^{k_2}_{\mathrm{hol}}(u_2)) = A \\
& = i^{k_1+1} \frac{(k_1+2)}{2N^{k_1+2}} (2\pi)^{k_1+1} \Lambda^*((G^{(1)}_{b_2,a_1}+G^{(1)}_{b_2,-a_1})(G^{(k_1+1)}_{b_1,-a_2} - G^{(k_1+1)}_{b_1,a_2}),0)
\end{align*}

\emph{Second cas : $k_2 \geq 1$.}

Puisque $\zeta(-\frac{b_2}{N},-k_2+1) = (-1)^{k_2} \zeta(\frac{b_2}{N},-k_2+1)$, on a $B=0$.

Comparons les termes $D$ et $E$. On a
\begin{align*}
E & = (-1)^{k_1} i \frac{(k_1+2)(k_2+2)(k+1)!}{8\pi N^2} \left(\hat{\zeta}(-\frac{a_1}{N},-k_2)+\hat{\zeta}(\frac{a_1}{N},-k_2)\right) \cdot \\
& \qquad \cdot \left(\hat{\zeta}(-\frac{b_1}{N},k+2)+(-1)^{k_1} \hat{\zeta}(\frac{b_1}{N},k+2) \right) \left(\hat{\zeta}(\frac{a_2}{N},k_2+2)-\hat{\zeta}(-\frac{a_2}{N},k_2+2)\right)
\end{align*}
D'après (\ref{hatzeta -x}), on a $\hat{\zeta}(-\frac{a_1}{N},-k_2)+\hat{\zeta}(\frac{a_1}{N},-k_2)=0$ si $k_2$ est pair. Donc $E$ est nul dès que $k_2$ est pair, et nous supposons désormais $k_2$ impair. Par la formule d'Hurwitz en $s=k_2+2$, on obtient
\begin{align*}
\zeta(-\frac{a_2}{N},-k_2-1) & = \frac{(k_2+1)!}{(2\pi)^{k_2+2}} \left( (-i)^{k_2+2} \hat{\zeta}(-\frac{a_2}{N},k_2+2)+i^{k_2+2} \hat{\zeta}(\frac{a_2}{N},k_2+2)\right) \\
& = \frac{i^{k_2+2} (k_2+1)!}{(2\pi)^{k_2+2}} \left(\hat{\zeta}(\frac{a_2}{N},k_2+2)-\hat{\zeta}(-\frac{a_2}{N},k_2+2)\right).
\end{align*}
Appliquons maintenant la formule d'Hurwitz aux termes à l'intérieur de la limite dans $D$. Il vient
\begin{align*}
\zeta(\frac{b_1}{N},s-k-1) & = \frac{\Gamma(k+2-s)}{(2\pi)^{k+2-s}} \left(e^{-\frac{i\pi}{2}(k+2-s)} \hat{\zeta}(\frac{b_1}{N},k+2-s) + e^{\frac{i\pi}{2}(k+2-s)} \hat{\zeta}(-\frac{b_1}{N},k+2-s)\right)\\
\zeta(-\frac{b_1}{N},s-k-1) & = \frac{\Gamma(k+2-s)}{(2\pi)^{k+2-s}} \left(e^{-\frac{i\pi}{2}(k+2-s)} \hat{\zeta}(-\frac{b_1}{N},k+2-s) + e^{\frac{i\pi}{2}(k+2-s)} \hat{\zeta}(\frac{b_1}{N},k+2-s)\right)
\end{align*}
d'où l'on déduit
\begin{align*}
& \zeta(\frac{b_1}{N},s-k-1) + (-1)^{k_1} \zeta(-\frac{b_1}{N},s-k-1) \\
& = \frac{\Gamma(k+2-s)}{(2\pi)^{k+2-s}} \left( e^{-\frac{i\pi}{2}(k+2-s)}(1-e^{-i\pi s}) \hat{\zeta}(\frac{b_1}{N},k+2-s) + e^{\frac{i\pi}{2}(k+2-s)}(1-e^{i\pi s}) \hat{\zeta}(-\frac{b_1}{N},k+2-s) \right)
\end{align*}
Lorsque $s \to 0$, on a $\Gamma(s-k_2-1) \sim \frac{1}{(k_2+1)!} s^{-1}$ et $1-e^{\pm i \pi s} \sim \mp i\pi s$ d'où
\begin{align*}
& \lim_{s \to 0} \left(\Gamma(s-k_2-1) \left(\zeta(\frac{b_1}{N},s-k-1) + (-1)^{k_1} \zeta(-\frac{b_1}{N},s-k-1)\right)\right) \\
& = \frac{(k+1)!}{(k_2+1)!(2\pi)^{k+2}} \left( (-i)^{k+2} i\pi \hat{\zeta}(\frac{b_1}{N},k+2) + i^{k+2} (-i\pi) \hat{\zeta}(-\frac{b_1}{N},k+2)\right)\\
& = \frac{(k+1)! (-i)^{k+2} i\pi}{(k_2+1)! (2\pi)^{k+2}} \left( \hat{\zeta}(\frac{b_1}{N},k+2) + (-1)^{k_1} \hat{\zeta}(-\frac{b_1}{N},k+2)\right).
\end{align*}
En reportant dans $D$, on obtient
\begin{align*}
D & = - i^{k_1} \frac{(k_1+2)(k_2+2)}{2N^2} (2\pi)^{k_1+2k_2+2} \left(\hat{\zeta}(\frac{a_1}{N},-k_2)+\hat{\zeta}(-\frac{a_1}{N},-k_2)\right) \cdot \\
& \qquad \cdot \frac{i^{k_2+2} (k_2+1)!}{(2\pi)^{k_2+2}} \left(\hat{\zeta}(\frac{a_2}{N},k_2+2)-\hat{\zeta}(-\frac{a_2}{N},k_2+2)\right) \cdot \\
& \qquad \cdot \frac{(k+1)! (-i)^{k+2} i\pi}{(k_2+1)! (2\pi)^{k+2}} \left( \hat{\zeta}(\frac{b_1}{N},k+2) + (-1)^{k_1} \hat{\zeta}(-\frac{b_1}{N},k+2)\right)\\
& = - i  \frac{(k_1+2)(k_2+2)(k+1)! }{8\pi N^2} \left(\hat{\zeta}(\frac{a_1}{N},-k_2)+\hat{\zeta}(-\frac{a_1}{N},-k_2)\right) \cdot \\
& \qquad \cdot \left(\hat{\zeta}(\frac{a_2}{N},k_2+2)-\hat{\zeta}(-\frac{a_2}{N},k_2+2)\right) \cdot \\
& \qquad \cdot \left( \hat{\zeta}(\frac{b_1}{N},k+2) + (-1)^{k_1} \hat{\zeta}(-\frac{b_1}{N},k+2)\right)\\
& = -E.
\end{align*}
Au final, on a donc dans ce cas
\begin{equation*}
\int_{X^k \{0,\infty\}}^{*} p_1^* \Eis^{k_1}_{\mathcal{D}}(u_1) \wedge \pi_{k_2+1}(p_2^* \Eis^{k_2}_{\mathrm{hol}}(u_2)) = A.
\end{equation*}

Dans tous les cas, on a donc
\begin{equation*}
\int_{X^k \{0,\infty\}}^* p_1^* \Eis^{k_1}_{\mathcal{D}}(u_1) \wedge \pi_{k_2+1} (p_2^* \Eis^{k_2}_{\hol}(u_2)) = A^{k_1,k_2}(u_1,u_2).
\end{equation*}

\begin{proof}[Démonstration du théorème \ref{main thm}] Rappelons que la forme différentielle $\Eis_{\mathcal{D}}^{k_1,k_2}(u_1,u_2)$ est donnée par

\begin{align*}
\Eis_{\mathcal{D}}^{k_1,k_2}(u_1,u_2) & = p_1^* \Eis^{k_1}_{\mathcal{D}}(u_1) \wedge \pi_{k_2+1} (p_2^* \Eis^{k_2}_{\hol}(u_2)) \\
& \qquad + (-1)^{k_1+1} \pi_{k_1+1} (p_1^* \Eis^{k_1}_{\hol}(u_1)) \wedge p_2^* \Eis^{k_2}_{\mathcal{D}}(u_2)\\
& = p_1^* \Eis^{k_1}_{\mathcal{D}}(u_1) \wedge \pi_{k_2+1} (p_2^* \Eis^{k_2}_{\hol}(u_2)) \\
& \qquad + (-1)^{(k_1+1)(k_2+1)} p_2^* \Eis^{k_2}_{\mathcal{D}}(u_2) \wedge \pi_{k_1+1} (p_1^* \Eis^{k_1}_{\hol}(u_1)).
\end{align*}
Notons $\theta : E^k \to E^k$ l'isomorphisme défini par
\begin{equation*}
\theta : E^k = E^{k_2} \times_{Y(N)} E^{k_1} \xrightarrow{\cong} E^{k_1} \times_{Y(N)} E^{k_2} = E^k.
\end{equation*}
Cet isomorphisme laisse stable le cycle de Shokurov $X^k \{0,\infty\}$ en multipliant l'orientation par $(-1)^{k_1 k_2}$. Il vient donc
\begin{align*}
& \int_{X^k \{0,\infty\}}^* p_2^* \Eis^{k_2}_{\mathcal{D}}(u_2) \wedge \pi_{k_1+1} (p_1^* \Eis^{k_1}_{\hol}(u_1)) \\
& = (-1)^{k_1 k_2} \int_{X^k \{0,\infty\}}^* \theta^* \left( p_2^* \Eis^{k_2}_{\mathcal{D}}(u_2) \wedge \pi_{k_1+1} (p_1^* \Eis^{k_1}_{\hol}(u_1)) \right) \\
& = (-1)^{k_1 k_2} A^{k_2,k_1}(u_2,u_1).
\end{align*}
On obtient ainsi
\begin{align*}
& \int_{X^k \{0,\infty\}}^{*} \Eis^{k_1,k_2}_{\mathcal{D}}(u_1,u_2) \\
& = A^{k_1,k_2}(u_1,u_2) + (-1)^{k_1+k_2+1} A^{k_2,k_1}(u_2,u_1) \\
& = \frac{(k_1+2)(k_2+2)}{4N^{k+2}} (2\pi)^{k+1} \left( i^{k_1-k_2+1}  \Lambda^*((G^{(k_2+1)}_{b_2,a_1}+G^{(k_2+1)}_{b_2,-a_1})(G^{(k_1+1)}_{b_1,-a_2} - G^{(k_1+1)}_{b_1,a_2}),0) \right. \\
& \qquad \left. +  (-1)^{k_1+k_2+1} i^{k_2-k_1+1}  \Lambda^*((G^{(k_1+1)}_{b_1,a_2}+G^{(k_1+1)}_{b_1,-a_2})(G^{(k_2+1)}_{b_2,-a_1} - G^{(k_2+1)}_{b_2,a_1}),0) \right)\\
& = \frac{(k_1+2)(k_2+2)}{4N^{k+2}} (2\pi)^{k+1} i^{k_1-k_2+1} \cdot \\
& \qquad \Lambda^*((G^{(k_2+1)}_{b_2,a_1}+G^{(k_2+1)}_{b_2,-a_1})(G^{(k_1+1)}_{b_1,-a_2} - G^{(k_1+1)}_{b_1,a_2}) - (G^{(k_1+1)}_{b_1,a_2}+G^{(k_1+1)}_{b_1,-a_2})(G^{(k_2+1)}_{b_2,-a_1} - G^{(k_2+1)}_{b_2,a_1}),0)\\
& = \frac{(k_1+2)(k_2+2)}{2N^{k+2}} (2\pi)^{k+1} i^{k_1-k_2+1} \Lambda^*( G^{(k_2+1)}_{b_2,a_1} G^{(k_1+1)}_{b_1,-a_2} - G^{(k_2+1)}_{b_2,-a_1} G^{(k_1+1)}_{b_1,a_2},0).
\end{align*}

\end{proof}

\bibliographystyle{smfplain}
\bibliography{references}

\end{document}